\documentclass[a4paper, 11pt]{article}
\usepackage{amsmath, amssymb, amsthm}
\usepackage{srcltx,graphicx}
\usepackage[multidot]{grffile}
\usepackage{color}
\usepackage{overpic}
\usepackage{multirow}
\usepackage{ltxtable}
\usepackage{hyperref}
\usepackage[hang]{subfigure}
\usepackage{mathdots}

\newtheorem{theorem}{Theorem}
\newtheorem{lemma}{Lemma}

\newcommand\bbR{\mathbb{R}}
\newcommand\bbN{\mathbb{N}}
\newcommand\bbI{\mathbb{I}}
\newcommand\bxi{\boldsymbol{\xi}}
\newcommand\bb{\boldsymbol{b}}
\newcommand\bc{\boldsymbol{c}}

\newcommand\bx{\boldsymbol{x}}

\newcommand\bu{\boldsymbol{u}}

\newcommand\be{\boldsymbol{e}}
\newcommand\bh{\boldsymbol{h}}
\newcommand\bn{\boldsymbol{n}}
\newcommand\br{\boldsymbol{r}}

\newcommand\bC{\boldsymbol{C}}

\newcommand\bS{\boldsymbol{S}}
\newcommand\bZ{\boldsymbol{Z}}
\newcommand\bA{\boldsymbol{\mathrm{A}}}

\newcommand\bH{\boldsymbol{\mathrm{H}}}
\newcommand\bI{\boldsymbol{\mathrm{I}}}
\newcommand\bK{\boldsymbol{\mathrm{K}}}

\newcommand\bM{\boldsymbol{\mathrm{M}}}
\newcommand\bQ{\boldsymbol{\mathrm{Q}}}
\newcommand\bR{\boldsymbol{\mathrm{R}}}
\newcommand\bW{\boldsymbol{\mathrm{W}}}

\newcommand\dd{\,\mathrm{d}}
\newcommand\He{\mathit{He}}
\newcommand\mH{\mathcal{H}}
\newcommand\Kn{\mathrm{Kn}}

\newcommand\bLambda{\boldsymbol{\mathrm{\Lambda}}}
\newcommand\bomega{\boldsymbol{\omega}}
\newcommand\bzero{\boldsymbol{0}}

\newcommand\diag{\mathrm{diag}}

\newcommand\pd[2]{\dfrac{\partial {#1}}{\partial {#2}}}

\newcommand\od[2]{\dfrac{\dd {#1}}{\dd {#2}}}
\newcommand\odd[2]{\dfrac{\mathrm{D}{#1}}{\mathrm{D}{#2}}}

\numberwithin{equation}{section}
\numberwithin{figure}{section}

{\theoremstyle{remark} \newtheorem{remark}{Remark}}

\graphicspath{{image/}}

\title{Resolving Knudsen Layer by High Order Moment Expansion}

\author{Yuwei Fan\thanks{Department of Mathematics, Stanford University, Stanford, CA 94305, 
    email: {\tt ywfan@stanford.edu}.},~~ 
Jun Li\thanks{School of Mathematical Sciences, Peking University,
    Beijing, China, email: {\tt lijun609@pku.edu.cn}.},~~
Ruo Li\thanks{CAPT, LMAM \& School of Mathematical Sciences, Peking
    University, Beijing, China, email: {\tt rli@math.pku.edu.cn}.},~~
Zhonghua Qiao\thanks{Department of Applied Mathematics, the Hong
    Kong Polytechnic University, Hung Hom, Hong Kong, 
    email:{\tt zhonghua.qiao@polyu.edu.hk}.}
}

\begin{document}
\maketitle
\begin{abstract}
  We model the Knudsen layer in Kramers' problem by linearized high
  order hyperbolic moment system. Due to the hyperbolicity, the
  boundary conditions of the moment system is properly reduced from
  the kinetic boundary condition. For Kramers' problem, we give the
  analytical solutions of moment systems. With the order increasing of
  the moment model, the solutions are approaching to the solution of
  the linearized BGK kinetic equation. The velocity profile in the
  Knudsen layer is captured with improved accuracy for a wide range of
  accommodation coefficients.

\vspace*{4mm}
\end{abstract}

\section{Introduction}
In the area of kinetic theory, Kramers' problem \cite{Kramers1949} is
generally considered as the most basic way to understand the
fundamental flow physics of the wall, which defining the Knudsen layer
\cite{Lilley2007}, without some of the additional complications in
other more realistic problems, such as flow in a plane channel
\cite{Garcia2009} or cylindrical tube \cite{Higuera1989,
Grucelski2013}. It is well known \cite{Karniadakis2002, Zhang2012}
that the classical Navier-Stokes-Fourier(NSF) equations with
appropriate boundary conditions can be used to describe the flow with
satisfactory accuracy when the gas is close to a statistical
equilibrium state. However, more accurate model is needed to depict
the nonequilibrium effects near the wall, where the continuum
assumption is essentially broken down and NSF equations themselves
become inappropriate \cite{Lilley2007, Dongari2009}. This is exactly
the case in Knudsen layers.

During the past decades, various methods have been developed to
investigate the Kramers' problem based on the Boltzmann equation. Highly
accurate results on the dependence of slip coefficient for the
unmodeled Boltzmann equation and general boundary condition have been
reported \cite{Loyalka1967, Loyalka1971, Klinc1972}. Variable
collision frequency models of the Boltzmann equation
\cite{Cercignani1969, Williams2001, Loyalka1967, Loyalka1975,
Loyalka1990, Siewert2001} are extensively discussed. We note that
the direct simulation Monte Carlo (DSMC) method \cite{Bird} is widely
used to solve the Boltzmann equation numerically. Unfortunately, DSMC
calculations impose prohibitive computational demands for many
applications of current interests. The intensive computational
demands of DSMC method have motivated recent interests in the
application of higher-order hydrodynamic models to simulate rarefied
flows \cite{Reese2003, Guo2006, Gu2010, Mizzi2007}. There are many
competing sets of higher-order constitutive relations, which are
derived from the fundamental Boltzmann equation using differing
approaches. The classical approaches are the Chapman-Enskog technique
and Grad's moment method. Among these alternative macroscopic modeling
and simulation strategies \cite{Grad, Levermore}, the moment method is
quite attractive due to its numerous advantages \cite{Muller,
Struchtrup2002, TorrilhonEditorial}. It is regarded as a useful
tool to extend classical fluid dynamics, and achieves highly accurate
approximations with great efficiency.

The moment method for gas kinetic theory \cite{Grad} has been applied
on wall-bounded geometries which supplemented by slip and jump
boundary conditions \cite{Marques2001}, while its application is
seriously limited due to the lack of hyperbolicity \cite{Muller,
Grad13toR13}. Particularly for the $3$D case, the moment system is
not hyperbolic in any neighborhood of the Maxwellian. Only recently
this fatal defect has been remedied \cite{Fan, Fan_new, ANRxx} that
globally hyperbolic models can be deduced. The global hyperbolicity of
the new models provides us the information propagation directions,
and thus a proper boundary condition of the moment model may be
proposed. This motivates us to study the Kramers' problem using the
new moment models.

Starting from the globally hyperbolic moment system (HME), we first
derive a linearized hyperbolic moment model to depict the Kramers'
problem. We found that the linearized model is even simpler than one's
expectation, since the equations for velocity are decoupled from
other equations in the system involving high order moments. The number
of equations in the decoupled part related with velocity is the same
as the moment expansion order only. Then we establish the boundary
conditions for the linearized moment model according to physical and
mathematical requirements for the system. Following Grad's approach in
\cite{Grad} for the kinetic accommodation model by Maxwell
\cite{Maxwell}, we propose the general boundary conditions for shear
flows. After that, by linearizing the velocity jump and high order
terms in the expression of the general boundary conditions, it is then
adapted to the boundary condition for the linearized model. This makes us
able to give the expression of velocity by solving the decoupled
system related with velocity together with the corresponding boundary
condition. It is extensively believed that the linearized system is
accurate enough for low-speed flows, which encourages us to apply the
solution of the velocity obtained to study Kramers' problem.

To obtain the full velocity profile and the velocity slip coefficient
in Kramers' problem, one may adopt a certain direct numerical method to
solve the linearized Boltzmann equation. However, the linearized moment
system can depict the velocity profile in the Knudsen layer with
analytical expressions. This can be used to provide a convenient
correction near the wall \cite{Lockerby2008} for the lower order
macroscopic system, such as NSF equations. In the moment method, the
Knudsen layer appears as superpositions of exponential layers
\cite{Struchtrup2008.1}. For the result we give based on HME, the
number of exponential layers is increasing. Comparing with the results
given by direct numerical simulation, our solutions illustrate a
significant improvement in accuracy than the results in references
when more and more high order moments are considered. Particularly,
our results can capture the velocity profile in the Knusen layer
accurately for a wide range of accommodation coefficients. We note that
our linearized model is of the same computational cost as the lower
order moment system.

This paper is organized as follows. In Section \ref{sec:hme} we
reviewed HME for Boltzmann equations and derived the linearized HME.
The boundary conditions for HME and linearized HME are established in
Section \ref{sec:bc}. The solutions of linearized equations are solved in
detail for Kramers' problem in Section \ref{sec:kramers}. With the
solutions of the velocity profile in Knudsen layer, some important
coefficients, such as defect velocity, are compared with the other
model of kinetic solution in the same section. We then draw some
conclusions to end this paper.


\section{Linearized HME for Boltzmann Equation}
\label{sec:hme}
\subsection{Boltzmann equation}
In gas kinetic theory, the motion of particles of gas can be depicted
by the Boltzmann equation \cite{Boltzmann}
\begin{equation}\label{eq:boltzmann}
  \pd{f}{t} + \bxi\cdot\nabla_{\bx}f = Q(f,f),
\end{equation}
where $f(t,\bx,\bxi)$ is the number density distribution function
which depends on the time $t\in\bbR^+$, the spatial position
$\bx\in\bbR^3$ and the microscopic particle velocity $\bxi\in\bbR^3$,
and $Q(f,f)$ is the collision term.  In this paper, we limit the
discussion on the BGK collision model \cite{BGK}, which reads:
\begin{equation}\label{eq:collision}
    Q(f,f) = \frac{\rho\theta}{\mu}(\mathcal{M} - f) ,
\end{equation}
where $\mu$ is the viscosity and $\mathcal{M}$ is the local
thermodynamic equilibrium, usually called the local Maxwellian,
defined by
\begin{displaymath}
    \mathcal{M}=\frac{\rho}{(2\pi\theta)^{3/2}}
    \exp\left( -\frac{|\bxi-\bu|^2}{2\theta} \right).
\end{displaymath}
Here the density $\rho$, the macroscopic velocity $\bu$ and the
temperature $\theta$ are related to the distribution function as
\begin{equation}
    \rho =\int_{\bbR^3}f\dd\bxi,\qquad
    \rho\bu =\int_{\bbR^3}\bxi f\dd\bxi,\qquad
    \rho|\bu|^2+3\rho\theta =\int_{\bbR^3}|\bxi|^2f\dd\bxi.
\end{equation}
Multiplying the Boltzmann equation by $(1,\bxi,|\bxi|^2)$ and
integrating both sides over $\bbR^3$ with respect to $\bxi$, we obtain
the conservation laws of mass, momentum and energy as
\begin{equation}\label{eq:conservationlaws}
  \begin{aligned}
    \odd{\rho}{t}&+\rho\sum_{d=1}^3\pd{u_d}{x_d}=0,\\ 
    \rho\odd{u_i}{t}&+\sum_{d=1}^3\pd{p_{id}}{x_d}=0,\\ 
    \frac{3}{2}\rho\odd{\theta}{t}&+\sum_{k,d=1}^3p_{kd}\pd{u_k}{x_d}+\sum_{d=1}^3\pd{q_d}{x_d}=0, 
  \end{aligned}
\end{equation}
where
$\odd{\cdot}{t} := \pd{\cdot}{t} + \displaystyle \sum_{d=1}^3
u_d \pd{\cdot}{x_d}$
is the material derivative, and the pressure tensor $p_{ij}$ and the
heat flux $q_i$ are defined by
\begin{equation}
  p_{ij}=\int_{\bbR^3}(\xi_i-u_i)(\xi_j-u_j)f\dd\bxi,\quad
  q_i=\frac{1}{2}\int_{\bbR^3}|\bxi-\bu|^2(\xi_i-u_i)f\dd\bxi, \quad
  i,j=1,2,3.
\end{equation}
For convenience, we define the pressure $p$ and the stress tensor
$\sigma_{ij}$ by
\[
p=\sum_{d=1}^3\frac{p_{dd}}{3}=\rho\theta,\quad 
\sigma_{ij}=p_{ij}-p\delta_{ij},\quad i,j=1,2,3.
\]

\subsection{HME and its linearization}
The moment method in kinetic theory is first proposed by Grad in 1949
\cite{Grad}. The primary idea is to expand the distribution function
around the Maxwellian into Hermite series
\begin{equation}\label{eq:expansion}
    f(t,\bx,\bxi) =
    \frac{\mathcal{M}}{\rho}\sum_{\alpha\in\bbN^3}f_{\alpha}(t,\bx)\He_{\alpha}^{[\bu,\theta]}(\bxi)=
    \sum_{\alpha\in\bbN^3}f_{\alpha}(t,\bx)\mH_{\alpha}^{[\bu,\theta]}(\bxi),
\end{equation}
where $\alpha = (\alpha_1, \alpha_2, \alpha_3) \in \bbN^3$ is a 3D
multi-index, and $\He_{\alpha}^{[\bu,\theta]}(\bxi)$ are generalized
Hermite polynomials defined by
\begin{equation}\label{eq:hermite-poly}
    \He_{\alpha}^{[\bu,\theta]}(\bxi) =
    \frac{(-1)^{|\alpha|}}{\mathcal{M}}
    \dfrac{\partial^{|\alpha|} \mathcal{M}}{\partial
        \xi_1^{\alpha_1} \partial \xi_2^{\alpha_2} \partial
        \xi_3^{\alpha_3}}, \qquad
        |\alpha|=\sum_{d=1}^3\alpha_d,
\end{equation}
and $\mH_{\alpha}^{[\bu,\theta]}(\bxi)$ is the basis function defined
by
\begin{equation}\label{eq:basis-fun}
    \mH_{\alpha}^{[\bu,\theta]}(\bxi) = \frac{\mathcal{M}}{\rho}
    \He_{\alpha}^{[\bu,\theta]}(\bxi).
\end{equation}
Directly calculations yield, for
$i,j=1,2,3$,
\begin{equation}
  \begin{aligned}
    &f_{0}=\rho,\quad
    f_{e_i}=0,\quad 
    \sum_{d=1}^3f_{2e_d}=0,\\
    p_{ij}=p\delta_{ij} &+ (1+\delta_{ij})f_{e_i+e_j},\quad
    q_i = 2 f_{3e_i}+\sum_{d=1}^3 f_{e_i+2e_d}.
  \end{aligned}
\end{equation}
Substituting Grad's expansion \eqref{eq:expansion} into the
Boltzmann equation, and matching the coefficient of the basis function
$\mH^{[\bu,\theta]}_{\alpha}(\bxi)$, one can obtain the governing equations of
$\bu$, $\theta$ and $f_{\alpha}$, $\alpha\in\bbN^3$. However, the
resulting system contains infinite number of equations. Choosing a
positive integer $3\leq M\in\bbN$, and discarding all the equations
including $\pd{f_{\alpha}}{t}$, $|\alpha|>M$, and setting
$f_{\alpha}=0$, $|\alpha|>M$ to closure the residual system, we obtain
the $M$-th order Grad's moment system. Since
\[
Q(f,f) = - \frac{p}{\mu} \sum_{|\alpha| \geq 2} f_{\alpha}
\mH_{\alpha}^{[\bu,\theta]}(\bxi)
=-\frac{p}{\mu}\mathrm{H}(|\alpha|-2)f_{\alpha},
\]
where $\mathrm{H}(n)$ is the Heaviside step function 
\[
\mathrm{H}(n) = \left\{ \begin{array}{ll}
  0,  &   n<0,\\
  1,  &   n\geq0,
\end{array} \right.
\]
the $M$-th order Grad's moment system can be written as
\begin{equation}\label{eq:arbit-system}
  \begin{aligned}
    \odd{f_{\alpha}}{t} &+ \sum_{d=1}^3 \left( \theta \pd{f_{\alpha-e_d}}{x_d} +
    (1-\delta_{|\alpha|,M})(\alpha_d + 1)\pd{f_{\alpha+e_d}}{x_d} \right) \\
    + \sum_{k=1}^3 f_{\alpha-e_k} \odd{u_k}{t} &+ \sum_{k,d=1}^3 \pd{u_k}{x_d}
    \left(\theta f_{\alpha-e_k-e_d} + (\alpha_d + 1) f_{\alpha-e_k+e_d}
    \right) \\
    + \frac{1}{2} \sum_{k=1}^3 f_{\alpha-2e_k} \odd{\theta}{t} &+ \sum_{k,d=1}^3
    \frac{1}{2} \pd{\theta}{x_d} \left(
    \theta f_{\alpha-2e_k-e_d} + (\alpha_d + 1) f_{\alpha-2e_k+e_d}
    \right)\\
    &= -\frac{p}{\mu}f_{\alpha}\mathrm{H}(|\alpha|-2) , \quad |\alpha| \leq M,
  \end{aligned}
\end{equation}
where $\delta_{i,j}$ is Kronecker delta. Here and hereafter we agree
that $(\cdot)_{\alpha}$ is taken as zero if any component of $\alpha$
is negative. 

However, as is pointed in \cite{Muller,Grad13toR13}, Grad's moment
system lacks global hyperbolicity and is not hyperbolic even in any
neighborhood of local Maxwellian.  The globally hyperbolic
regularization proposed in \cite{Fan,Fan_new} figures the drawback out
essentially, and results in globally Hyperbolic Moment Equations (HME)
as 
\begin{equation}\label{eq:moment-system}
  \begin{aligned}
    \odd{f_{\alpha}}{t} &+ \sum_{d=1}^3 \left( \theta \pd{f_{\alpha-e_d}}{x_d} +
    (1 - \delta_{|\alpha|,M})(\alpha_d + 1)\pd{f_{\alpha+e_d}}{x_d} \right) \\
    + \sum_{k=1}^3 f_{\alpha-e_k} \odd{u_k}{t} &+ \sum_{k,d=1}^3 \pd{u_k}{x_d}
    \left(\theta f_{\alpha-e_k-e_d} + (1 - \delta_{|\alpha|,M})(\alpha_d + 1) f_{\alpha-e_k+e_d}
    \right) \\
    + \frac{1}{2} \sum_{k=1}^3 f_{\alpha-2e_k} \odd{\theta}{t} &+ \sum_{k,d=1}^3
    \frac{1}{2} \pd{\theta}{x_d} \left(
    \theta f_{\alpha-2e_k-e_d} + (1 - \delta_{|\alpha|,M})(\alpha_d + 1) f_{\alpha-2e_k+e_d}
    \right)\\
    &= -\frac{p}{\mu}f_{\alpha}\mathrm{H}(|\alpha|-2) , \quad |\alpha| \leq M.
  \end{aligned}
\end{equation}

Next, we try to derive the linearized system of
\eqref{eq:moment-system}. This requires us to examine the case that
the distribution function is in a small neighborhood of an equilibrium
state
\[
  f_0(\bxi) = \frac{\rho_0}{(2\pi\theta_0)^{\frac{3}{2}}} \mathrm{exp}
  \left( - \frac{|\bxi|^2}{2\theta_0}\right) ,
\]
given by $\rho_0, \theta_0, \bu = 0$. We introduce the dimensionless
variables $\bar{\rho}$, $\bar{\theta}$, $\bar{\bu}$, $\bar{p}$,
$\bar{p}_{ij}$ and $\bar{f}_{\alpha}$ as
\begin{equation}\label{eq:dimensionless}
  \begin{aligned}
    &\rho = \rho_0 (1 + \bar{\rho}),\quad \bu = \sqrt{\theta_0} \bar{\bu},
    \quad \theta = \theta_0 (1 + \bar{\theta}),
    \quad p = p_0 (1 + \bar{p}),\\
    &p_{ij}=p_0(\delta_{ij}+\bar{p}_{ij}),
    \quad f_{\alpha}=\rho_0\theta_0^{\frac{|\alpha|}{2}} \cdot \bar{f}_{\alpha},
    \quad \bx = L\cdot
    \bar{\bx},\quad t = \frac{L}{\sqrt{\theta_0}}\bar{t},
  \end{aligned}
\end{equation}
where $L$ is a characteristic length, $\bar{\bx}$ and $\bar{t}$ are
the dimensionless coordinates and time, respectively. Assume all the
dimensionless variables $\bar{\rho}$, $\bar{\theta}$, $\bar{\bu}$,
$\bar{p}$, $\bar{p}_{ij}$ and $\bar{f}_{\alpha}$ are small quantities.
Substituting \eqref{eq:dimensionless} into the globally hyperbolic
moment system \eqref{eq:moment-system}, and discarding all the
high-order small quantities, and noticing that $u_d \pd{\cdot}{x_d}$
is high-order small quantity, $\odd{\cdot}{t}\approx\pd{\cdot}{t}$, we
obtain the linearized HME as
\begin{equation}\label{eq:linear-system}
  \begin{aligned}
    &\pd{\bar{\rho}}{\bar{t}} + \sum_{d=1}^3\pd{\bar{u}_d}{\bar{x}_d} = 0,\\
    &\pd{\bar{u}_k}{\bar{t}} + \pd{\bar{p}}{\bar{x}_k} +
    \sum_{d=1}^3\pd{\bar{\sigma}_{kd}}{\bar{x}_d} = 0,\\
    &\pd{\bar{p}_{ij}}{\bar{t}} + \sum_{d=1}^3\delta_{ij}\pd{\bar{u}_d}{\bar{x}_d} + 
    \pd{\bar{u}_j}{\bar{x}_i} +
    \pd{\bar{u}_i}{\bar{x}_j} + \sum_{d=1}^3(e_i + e_j + e_d)! 
    \pd{\bar{f}_{e_i + e_j + e_d}}{\bar{x}_d} =
    -\frac{\bar{\sigma}_{ij}}{{\Kn}},\\
    &\begin{split}
    &\pd{\bar{f}_{\alpha}}{\bar{t}} + \sum_{d=1}^3\pd{\bar{f}_{\alpha - e_d}}{\bar{x}_d}
    + \sum_{d=1}^3(\alpha_d + 1)(1 - \delta_M)
    \pd{\bar{f}_{\alpha + e_d}}{\bar{x}_d} \\
	    &\qquad\qquad\qquad\qquad+ \sum_{d=1}^3\frac{1}{2}\delta_{\alpha,e_d+2e_k}
    \pd{\bar{\theta}}{\bar{x}_d}
    =-\frac{\bar{f}_{\alpha}}{{\Kn}},\quad 3 \leq |\alpha|\leq M,
    \end{split}
  \end{aligned}
\end{equation}
where $\bar{\sigma}_{ij}=\bar{p}_{ij}-\bar{p}\delta_{ij}$,
$i,j=1,2,3$, and $\delta_{\alpha, e_d+2e_k}$ is $1$ iff
$\alpha=e_d+2e_k$, otherwise is $0$. 
The Knudsen number $\Kn$ is defined by 
\begin{displaymath}
    \Kn = \frac{\lambda}{L},
\end{displaymath}
where $\lambda = \frac{\mu}{p_0}\sqrt{\theta_0}$ is the mean free
path.


\section{Boundary Condition}\label{sec:bc}
In this paper, we adopt Maxwell's accommodation boundary condition
\cite{Maxwell}, which is the most commonly used boundary condition in
gas kinetic theory. It is formulated as a linear combination of the
specular
reflection and the diffuse reflection. Wall boundary only requires the
incoming half of the distribution function when $\bxi \cdot \bn > 0$,
where $\bn$ is the unit normal vector pointing into the gas. With the
given velocity $\bu^W(t,\bx)$ and temperature $\theta^W(t,\bx)$ of the
wall, the boundary condition at the wall is
\begin{equation}\label{eq:Maxwell}
  f^W(t, \bx, \bxi) = 
  \left \{
  \begin{array}{ll}
    \chi f^W_M(t, \bx, \bxi) +(1 - \chi)f(t, \bx, \bxi^{\ast}),
    & \bC^W \cdot \bn > 0, \\
    f(t, \bx, \bxi), & \bC^W \cdot \bn \leq 0,
  \end{array}
  \right.
\end{equation}
where
\begin{equation}\label{eq:equilibrium}
  \begin{aligned}
    \bxi^{\ast} = \bxi - 2(\bC^W \cdot \bn)\bn, 
    \quad \bC^W = \bxi - \bu^W(t,\bx) ,\\ 
    f^W_M(t, \bx, \bxi) =
    \frac{\rho^W(t,\bx)}{(2\pi\theta^W(t,\bx))^{3/2}}
    \exp\left(-\frac{|\bxi -
      \bu^W(t,\bx)|^2}{2\theta^W(t,\bx)}\right),   
  \end{aligned}
\end{equation}
and $\chi \in [0,1]$ is the accommodation coefficient.

A boundary condition for general hyperbolic moment system was proposed
in \cite{Li}, which is derived from the Maxwell boundary condition by
calculating the expression of the moments at the wall. Here we are
purposely considering only steady shear flow, thus we adopt an
alternative approach to derive our boundary conditions. Let the unit
normal vector of the wall $\bn = (0,1,0)^T$. The velocity of the wall
$\bu^W = (u^W, 0, 0)$, and velocity for steady shear flow is
$\bu = (u_1, 0, 0)$. For $\bxi^{\ast} = (\xi_1,-\xi_2,\xi_3)$,
\eqref{eq:Maxwell} is precisely as
\begin{equation}\label{eq:wall-function}
  f^W(\bx,\bxi) = \left \{
  \begin{array}{ll}
    \chi f^W_M(\bx, \bxi) +(1 - \chi)f(\bx, \bxi^{\ast}), 
    & \xi_2 > 0, \\
    f(\bx, \bxi), & \xi_2 \leq 0.
  \end{array}
  \right .
\end{equation}
Denote $\Omega = \{\bxi \in \bbR^3\}$,
$\Omega^+ = \{\xi_1 \in \bbR, \xi_2 \in \bbR^+, \xi_3 \in \bbR \}$ and
$\Omega^- = \{\xi_1 \in \bbR, \xi_2 \in \bbR^-, \xi_3 \in \bbR \}$.
The integral of the wall distribution function
\eqref{eq:wall-function} with any function $\psi(\bC)$ gives us an
equation
\begin{equation}\label{eq:integral-equation}
  \begin{aligned}
    \int_{\Omega}\psi(\bC) f^W(\bx, \bxi) & \dd\bxi = 
    \int_{\Omega^-} \psi(\bC) f(\bx, \bxi) \dd\bxi \\ 
    &+\int_{\Omega^+}\psi(\bC)\left(\chi f_M^W(\bx, \bxi-\bu^W) + (1 -
    \chi) f(\bx, \bxi^{\ast})\right) \dd\bxi, 
  \end{aligned}
\end{equation}
where $\bC = (\xi_1-u_1, \xi_2, \xi_3)$. 

Definitely, for HME one has to restrict the form of function
$\psi(\bC)$, otherwise \eqref{eq:integral-equation} will produce too
many boundary conditions. It is clear that we should restrict
ourselves to those $\psi$'s that the moments in the equation can be
retrieved. Thus those $\psi$'s are polynomials as $\bC^{\beta}$,
$|\beta\leq M$, where $\beta = (\beta_1, \beta_2, \beta_3) \in
\bbN^3$ is a 3D multi-index. Moreover, the distribution function of
shear flow is an
even function in the $\xi_3$ direction, which leads to
$f_{\beta} = 0$, for $\beta_3$ is odd. Following Grad's theory
\cite{Grad} to limit the number of boundary condition in order to
ensure the continuity of boundary conditions when $\chi \to 0$, only a
subset of all the moments corresponding to 
\begin{equation}\label{set:bbI}
\{\bC^{\beta} \big| \beta \in \bbI \}, \qquad \text{where}\qquad
\bbI = \{|\beta| \leq M~\big|~\beta_2~\text{is odd and}~\beta_3
~\text{is even}\}
\end{equation}
can  be used to construct the wall boundary conditions. Then
we reformulate the equation \eqref{eq:integral-equation} as
\begin{equation}\label{eq:bc-xi2}
  \int_{\Omega^+}\bC^{\beta}f^W_M(\bx, \bxi-\bu^W) \dd\bxi =
  \frac{1}{\chi}\left( \int_{\Omega^+}\bC^{\beta}\left(f(\bx, \bxi) -
  (1-\chi)f(\bx, \bxi^{\ast})\right)\dd\bxi\right), 
  \quad \beta\in\bbI.
\end{equation}
Notice that the basis function defined in \eqref{eq:basis-fun} is
decoupled in compoents of $\bxi$. We then substitute
\eqref{eq:expansion} into \eqref{eq:bc-xi2} to calculate the integral
on both left and right hand side in \eqref{eq:bc-xi2},
respectively. To give the results, we first make some simplification
and define the following notations. Let 
\begin{equation}
	J_0(u,\theta) = 1,\quad 
	J_1(u,\theta) = u,\quad 
	J_{k+1}(u,\theta)= u J_k(u,\theta)+k\theta J_{k-1},
	k\geq 1,
\end{equation}
then
\[
  \frac{1}{\sqrt{2\pi\theta^W}}\int_{-\infty}^{\infty}(\xi_1 - u_1)^k
  \exp\left(-\frac{|\xi_1-u_1^W|^2}{2\theta^W}\right)\dd\xi_1
  =J_k(u_1^W-u_1,\theta^W).
\]
Let 
\[
    K(k,m) := \int_{-\infty}^{\infty}\frac{1}{\sqrt{2\pi}}\exp\left(
    -\frac{|\xi|^2}{2} \right)\xi^k \He_m(\xi)\dd\xi,
\]
where $\He_m(\xi)$ is $m$-th Hermite polynomial, then using the
orthogonal relation of the Hermite polynomials, one can find 
$K(0,m) = \delta_{0,m}$.
Denote the half space integral by
\begin{equation}
    S^\star (k,m) := \int_0^{\infty}\xi^k
    \He_m(\xi)\exp\left(-\frac{\xi^2}{2}\right)\dd\xi,
\end{equation}
and we have the following properties for $S^\star (k,m)$.
\begin{itemize}
\item Recursion relation:
  \begin{equation}\label{eq:rec}
    S^{\star}(k, m) = (k - 1) S^{\star}(k-2, m) + m S^{\star}(k-1, m-1).
  \end{equation}
\item The value of $S^\star (k,m)$ is:
  \begin{enumerate}
    \item If $m \leq k$:
      \begin{enumerate}
      \item If $k - m$ is even, $S^{\star}(k, m) = \sqrt{2 \pi}\cdot A$;
      \item If $k - m$ is odd, $S^{\star}(k, m) = B$;
      \end{enumerate}
      here $A$ and $B$ are two algebraic numbers.
    \item If $m > k$ and $k - m$ is even, $S^{\star}(k, m) = 0$.
  \end{enumerate}
\end{itemize}
Let
\begin{equation}\label{eq:integralS}
    \begin{aligned}
        S(k,m) &:= \frac{\hat{\chi}}{\sqrt{2\pi}} S^\star(k,m) \\
        &\ = \frac{\theta^{(m-k)/2}}{\chi}
        \int_0^{\infty}\xi^k\left(\He_m(\xi)-
        (1-\chi)\He_m(-\xi)\right)\exp\left( -\frac{|\xi|^2}{2} \right)\dd\xi,
    \end{aligned}
\end{equation}
where
\[
    \hat{\chi} = \left\{ 
        \begin{array}{ll}
            1, & m ~\text{is even},\\
            \frac{2-\chi}{\chi}, & m ~\text{is odd},
        \end{array} \right.
\]
then for each $\beta \in \bbI$ in \eqref{set:bbI}, the left and right
hand side of \eqref{eq:bc-xi2} can be represented by
\begin{equation}\label{eq:bc}
    \begin{aligned}
        \text{lhs of \eqref{eq:bc-xi2}}
        &=\frac{\rho^W \left(\theta^W\right)^{(\beta_2+\beta_3)/2}}{\sqrt{2\pi}}
        J_{\beta_1}\left(u_1^W-u_1,\theta^W\right)
        (\beta_2-1)!!(\beta_3-1)!!,\\
        \text{rhs of \eqref{eq:bc-xi2}} &= 
        \sum_{\alpha\in\bbN^3} \left(
        K(\beta_1,\alpha_1) 
        S(\beta_2,\alpha_2)
        K(\beta_3,\alpha_3)
        \theta^{(\beta_2-\alpha_2)/2}
        \right) f_{\alpha}.
    \end{aligned}
\end{equation}
Noticing $K(0,m)=\delta_{0,m}$, by setting $\beta=e_2$ in
\eqref{eq:bc}, we have
\begin{equation}\label{eq:e2bc}
  \rho^W\sqrt{\frac{\theta^W}{2\pi}} = \sum_{m=0}^{\infty}
  S(1,m)\frac{f_{me_2}}{{\theta}^{(m-1)/2}}.
\end{equation}
Let $p_{w} = \rho^W\sqrt{\theta^W \theta}$, then we have
\begin{equation}
    p_{w} =
    \sqrt{2\pi\theta}\sum_{m=0}^{\infty}S(1,m)\frac{f_{me_2}}{\theta^{(m-1)/2}}
    =p + f_{2e_2} - \frac{f_{4e_2}}{\theta} + \frac{3}{\theta^2}
    f_{6e_2} - \frac{15}{\theta^3}f_{8e_2} + \cdots.
\end{equation}
The boundary condition for the case $\beta = e_1+\beta_2e_2 \in \bbI$
in \eqref{set:bbI} is
\begin{equation}\label{eq:beta2oddbc}
  \frac{\rho^W}{\sqrt{2\pi}}(\theta^W)^{\frac{\beta_2}{2}}(u_1^W - u_1)
  (\beta_2 -1)!! = \sum_{\alpha_2} S(\beta_2, 
  \alpha_2) f_{e_1+ \alpha_2e_2} \theta^{(\beta_2-\alpha_2)/2}.
\end{equation}
Particularly, for the case $\beta = e_1+e_2$, one has
\[
  p_{w}\sqrt{\frac{\theta^W}{2\pi\theta}}(u_1^W - u_1) =
  S(1,1) \sigma_{12} + \sum_{\alpha_2 > 1}
  S(1,\alpha_2) f_{e_1+\alpha_2e_2}\theta^{(1-\alpha_2)/2}. 
\]

Here we only consider the boundary condition for the specific case
that $\beta = e_1+\beta_2e_2 \in \bbI$ in \eqref{eq:beta2oddbc},
which is
\begin{equation}\label{eq:bc-palpha}
    p_{w}
    \frac{\left(\theta^W\right)^{\frac{\beta_2-1}{2}}}{\sqrt{2\pi}}(\beta_2 -1)!!
    (u_1^W - u_1) = 
    \sum_{\alpha_2} \theta^{\frac{1+\beta_2-\alpha_2}{2}}
    S(\beta_2, \alpha_2) f_{e_1+\alpha_2e_2}.
\end{equation}
We linearize this condition at $\theta_0$ as that in
\eqref{eq:dimensionless} for our purpose, and assume
$\theta^W-\theta_0$ is a small quantity. By substituting
\eqref{eq:dimensionless} into \eqref{eq:bc-palpha}, and applying the
closure of HME, i.e $f_{\alpha}=0$, $|\alpha|>M$, the linearized
boundary condition is arrived at as
\begin{equation}\label{bc:linear}
  \frac{(\beta_2 -1)!!}{\sqrt{2\pi}}(\bar{u}_1^W - \bar{u}_1) = \sum_{\alpha_2\leq M}
  S(\beta_2,\alpha_2) \bar{f}_{e_1+\alpha_2e_2},
\end{equation}
where $\bar{u}_1^W$ is defined as dimensionless variable $u_1^W =
\sqrt{\theta_0}\bar{u}_1^W$, and $\beta_2$ is odd and $|\beta_2|\leq
M$. 




\section{Kramers' Problem}
\label{sec:kramers}
Our setup for Kramers' problem is standard. The gas flow in a
half-space over a flat wall is considered, and the coordinates are
chosen such that $x$ direction is parallel to the wall, and $y$
direction is perpendicular to the wall. The solid wall is fixed on
$\bar{y} = 0~(\bar{u}_1^W = 0)$.  The temperature and density of the
gas far from the wall are constant. Gas velocity is $\bar{\bu} =
(\bar{u}_1, 0, 0)$ and all derivatives in equations
\eqref{eq:linear-system} in $x$ and $z$ direction are zero. 

\subsection{Formal solution of linearized HME}
We give the formal solution of the linearized HME at first. The setup
of Kramers' problem makes the equations of linearized moment system
\eqref{eq:linear-system} related to velocity decoupled from the whole
linearized moment system, which enables us to investigate the velocity
by studying a small system as
\begin{equation}\label{eq:velocity}
  \begin{aligned}
    &\od{\bar{\sigma}_{12}}{\bar{y}} = 0, \\
    &\od{\bar{u}_1}{\bar{y}} + 2\od{\bar{f}_{e_1+2e_2}}{\bar{y}}
    = -\frac{1}{\Kn}\bar{\sigma}_{12}, \\
    &\od{\bar{\sigma}_{12}}{\bar{y}} +
    3\od{\bar{f}_{e_1+3e_2}}{\bar{y}} =
    -\frac{1}{\Kn}\bar{f}_{e_1+2e_2}, \\ 
    &\cdots \\
    &\od{\bar{f}_{e_1+ (M-2)e_2}}{\bar{y}}
    = -\frac{1}{\Kn}\bar{f}_{e_1+(M-1)e_2}. 
  \end{aligned}
\end{equation}
We collect the variables involved in \eqref{eq:velocity} into a vector
\[
    V = \left(\bar{u}_1, \bar{\sigma}_{12}, \bar{f}_{e_1+2e_2},
    \bar{f}_{e_1+3e_2},\cdots, \bar{f}_{e_1+(M-1)e_2}\right)^T,
\] 
and then \eqref{eq:velocity} is formulated as
\begin{equation}\label{eq:simple-velocity}
  \bM \od{V}{\bar{y}} = -\frac{1}{\Kn}\bQ V,
\end{equation}
where
\begin{equation}\label{eq:def_MQ}
\bM = \left(
  \begin{array}{cccccc}
    0 & 1 &        &        &        & \\
    1 & 0 & 2      &        &        & \\
      & 1 & 0      & 3      &        & \\
      &   & \ddots & \ddots & \ddots & \\
      &   &        & 1      & 0      & M - 1\\
      &   &        &        & 1      & 0
  \end{array}
  \right), \quad \bQ = \left(
    \begin{array}{cccc}
      0 &   &        &  \\
        & 1 &        &  \\
        &   & \ddots &  \\
        &   &        & 1
    \end{array}
  \right).
\end{equation}
Easy to check that the matrix $\bM$ is real diagonalizable.
Actually, we have the eigen-decomposition of $\bM$
as $\bM = \bR \bLambda \bR^{-1}$, where $\bR$ is the Hermite transformation matrix
\begin{equation}\label{mat:eigen-vec}
    \bR = (r_{ij})_{M\times M},\quad
    r_{ij}=\frac{\He_{i-1}(\lambda_j)}{(i-1)!},
    \quad i,j=1,\cdots,M,
\end{equation}
and $\bLambda = \diag\{\lambda_i; i = 1, \cdots, M\}$,
where the eigenvalues $\lambda_i$, $i=1,\cdots,M$ are zeros of the
$M$-th order Hermite polynomial $\He_M(x)$. We sort the eigenvalues
$\lambda_i$ in decending order, saying $\lambda_i >
\lambda_{i+1}$. The diagonal matrix $\bLambda$ can then be written as
\begin{equation}\label{mat:Lamb}
\bLambda = \left(
  \begin{array}{cc}
    \bLambda_+ & \\
    & \bLambda_{\leq 0}
  \end{array}
\right),
\end{equation}
\begin{equation}\label{mat:pos-neg}
\begin{aligned}
& \bLambda_+ = \diag \left\{\lambda_i;~  i =
  1, \cdots, \lfloor \frac{M}{2} \rfloor \right\}, \\
& \bLambda_{\leq 0} = \diag \left\{\lambda_i;~  i =
  \lfloor \frac{M}{2} \rfloor + 1, \cdots, M \right\}.
\end{aligned}
\end{equation}

The first equation of \eqref{eq:simple-velocity} indicates
$\bar{\sigma}_{12}$ is a constant and the second equation of
\eqref{eq:simple-velocity} gives that
\[
    \bar{u}_1(\bar{y}) = -\dfrac{\bar{y}}{\Kn} \bar{\sigma}_{12} -
    2\bar{f}_{e_1+2e_2}(\bar{y}) + c_0, 
\]
where $c_0$ is a constant to be determined. We denote
\[
\hat{V} = (\bar{f}_{e_1+2e_2}, \bar{f}_{e_1+3e_2}, \cdots,
\bar{f}_{e_1+(M-1)e_2})^T, 
\] 
which is the remaining part of $V$ excluded the first two variables
$\bar{u}_1$ and $\bar{\sigma}_{12}$. Then the system with higher order
moments is separated from \eqref{eq:simple-velocity}, which reads as
\begin{equation}\label{eq:matrix-hatM}
  \hat{\bM} \od{\hat{V}}{\bar{y}} = -\frac{1}{\Kn} \hat{V},
\end{equation}
where 
\[
  \hat{\bM} = \left(
  \begin{array}{cccccc}
    0 & 3 &        &        &        & \\
    1 & 0 & 4      &        &        & \\
      & 1 & 0      & 5      &        & \\
      &   & \ddots & \ddots & \ddots & \\
      &   &        & 1      & 0      & M - 1\\
      &   &        &        & 1      & 0
  \end{array}
  \right).
\]

Correspondingly to the matrix $\bM$, the matrix $\hat{\bM}$ is real
diagonalizable, too.  Precisely, let
\[
    \hat{\He}_0(x) = 1,~\hat{\He}_1(x)=x,~
    \hat{\He}_{k+1}(x)=x\hat{\He}_k(x)-(k+2)\hat{\He}_{k-1}(x),
    ~k\geq1,
\]
then the characteristic polynomial of $\hat{\bM}$ is
$\hat{\He}_{M-2}(\lambda)$. The recursion relation implies that
$\hat{\He}_k(x)$ has $k$ real and simple zeros, and thus $\hat{\bM}$
is real diagonalizable and the eigenvalues $\hat{\lambda}_i$,
$i = 1, \cdots, M-2$, are the zeros of $\hat{\He}_{M-2}
(\lambda)$.
Furthermore, if $\hat{\lambda}_i$ is an eigenvalue of $\hat{\bM}$,
then $-\hat{\lambda}_i$ has to be an eigenvalue of $\hat{\bM}$, since
$\hat{\He}_{M-2}(x)$ is an odd function if $M$ is odd and is an even
function if $M$ is even. As for the matrix $\bM$, we sort the
eigenvalues $\hat{\lambda}_i$ in decending order, too, to make the
diagonal matrix
\[
\hat{\bLambda} = \left(
  \begin{array}{cc}
    \hat{\bLambda}_+ & \\
    & \hat{\bLambda}_{\leq 0}
  \end{array}
  \right),
\]
\[
\begin{aligned}
& \hat{\bLambda}_+ = \diag \left\{ \hat{\lambda}_i;~ i = 1, \cdots,
  \lfloor \frac{M}{2} \rfloor - 1 \right\}, \\
& \hat{\bLambda}_{\leq 0} = \diag \left\{ \hat{\lambda}_i;~ i = 
  \lfloor \frac{M}{2} \rfloor, \cdots, M - 2 \right\}.
\end{aligned}
\]
Then eigen-decomposition of $\hat{\bM}$ is
$\hat{\bM} = \hat{\bR} \hat{\bLambda} \hat{\bR}^{-1}$, where
\begin{equation}\label{eq:def_hatbR}
    \hat{\bR}=(\hat{r}_{ij})_{(M-2)\times(M-2)},\quad 
    \hat{r}_{ij}=\frac{\hat{\He}_{i-1}(\hat{\lambda}_j)}{(i+1)!},
    \quad i,j=1,\cdots,M-2.
\end{equation}
Let us define the matrix $\hat{\bR}_+$ as the left
$\lfloor \dfrac{M}{2} \rfloor - 1$ colomns of $\hat{\bR}$,
$\hat{\bR}_-$ as the right $\lceil \dfrac{M}{2} \rceil - 1$ colomns of
$\hat{\bR}$ for latter usage, and precisely, we have
\[
\hat{\bR}_+ = \left(    
  \begin{array}{ccc}
    \frac{\hat{\He}_0(\hat{\lambda}_1)}{2!} & \cdots 
    & \frac{\hat{\He}_0(\hat{\lambda}_{\lfloor \frac{M}{2} \rfloor - 1})}{2!} \\
    \vdots & \ddots & \vdots \\
    \frac{\hat{\He}_{M-3}(\hat{\lambda}_1)}{(M-1)!} & \cdots 
    & \frac{\hat{\He}_{M-3}(\hat{\lambda}_{\lfloor \frac{M}{2} \rfloor - 1})}{(M-1)!}
  \end{array}\right), \quad
  \hat{\bR}_- = \left(    
  \begin{array}{ccc}
    \frac{\hat{\He}_0(\hat{\lambda}_{\lfloor \frac{M}{2} \rfloor})}{2!} & \cdots 
    & \frac{\hat{\He}_0(\hat{\lambda}_{M-2})}{2!} \\
    \vdots & \ddots & \vdots \\
    \frac{\hat{\He}_{M-3}(\hat{\lambda}_{\lfloor \frac{M}{2} \rfloor})}{(M-1)!} & \cdots 
    & \frac{\hat{\He}_{M-3}(\hat{\lambda}_{M-2})}{(M-1)!}
  \end{array}\right).
\]
Let $\hat{\bR}_{+,\text{even}}$, which is made with the even rows of
$\hat{\bR}_+$ as
\[
\begin{aligned}
  \hat{\bR}_{+,\text{even}} &\triangleq
  (\hat{r}_{ij}), \text{~where~} i \text{~is
    even,~} j = 1,\cdots,\lfloor \frac{M}{2} \rfloor - 1, \\
  &=\left(
    \begin{array}{cccc}
      \dfrac{\hat{\He}_1(\hat{\lambda}_1)}{3!} 
      & \dfrac{\hat{\He}_1(\hat{\lambda}_2)}{3!}
      & \cdots
      & \dfrac{\hat{\He}_1(\hat{\lambda}_{\lfloor \frac{M}{2} \rfloor - 1})}{3!} \\
      \dfrac{\hat{\He}_3(\hat{\lambda}_1)}{5!} 
      & \dfrac{\hat{\He}_3(\hat{\lambda}_2)}{5!}
      & \cdots
      & \dfrac{\hat{\He}_3(\hat{\lambda}_{\lfloor \frac{M}{2} \rfloor - 1})}{5!} \\
      \vdots & \vdots & \cdots & \vdots \\
    \end{array}
  \right),
\end{aligned}
\]    
be a $\lfloor \frac{M}{2} \rfloor - 1 \times \lfloor \frac{M}{2}
\rfloor - 1$ square matrix. And corresponding to
$\hat{\bR}_{+,\text{even}}$, define $\hat{\bR}_{+,\text{odd}},
\hat{\bR}_{-,\text{odd}}, \hat{\bR}_{-,\text{odd}}$ as
\[
\begin{aligned}
  \hat{\bR}_{+,\text{odd}} &\triangleq
  (\hat{r}_{ij}), \text{~where~} i \text{~is
    odd,~} j = 1,\cdots,\lfloor \frac{M}{2} \rfloor - 1, \\
  \hat{\bR}_{-,\text{even}} &\triangleq
  (\hat{r}_{ij}), \text{~where~} i \text{~is
    even,~} j = \lfloor \frac{M}{2} \rfloor,\cdots,M-2, \\
  \hat{\bR}_{-,\text{odd}} &\triangleq
  (\hat{r}_{ij}), \text{~where~} i \text{~is
    odd,~} j = \lfloor \frac{M}{2} \rfloor,\cdots,M-2.
\end{aligned}
\]
We declare that
\begin{lemma}\label{lem:R_peven_invertible}
  $\hat{\bR}_{+,\text{even}}$ is invertible.
\end{lemma}
\begin{proof}
  Let $\boldsymbol{P}_\sigma$ to be the permutation matrix of the
  permutation
  \begin{equation}\label{mat:P_sigma}
    \sigma: \left\{ 1, 2, \cdots , M-2 \right\} \rightarrow \left\{ 1,
      2, \cdots , M-2 \right\},
  \end{equation}
  that 
  \[
  \sigma(i) = \mod(i,2) \times (\lfloor \frac{M}{2} \rfloor - 1) +
  \lfloor i/2 \rfloor. 
  \]
  The permutation maps the list of numbers $1, 2, \cdots, M -2 $ to
  \[
  2, 4, 6, \cdots, 1, 3, 5, \cdots
  \]
  that the even numbers are ahead of the odd numbers.
  Then matrix $\hat{\bR}$ is re-organized by the
  permutation matrix as
  \[
  \boldsymbol{P}^{-1}_\sigma \hat{\bR} = \left(
    \begin{array}{c|c}
      \hat{\bR}_{+,\text{even}} & \hat{\bR}_{-,\text{even}}\\
      \hat{\bR}_{+,\text{odd}}  & \hat{\bR}_{-,\text{odd}}
    \end{array}
  \right).
  \]
  Notice that each eigenvalue $\hat{\lambda}_i \in \hat{\bLambda}_+$,
  $-\hat{\lambda}_i \in \hat{\bLambda}_{\leq 0}$. Then for any
  eigenvector
  \[
  \hat{\br}_i = (\hat{\br}_{i,\text{even}} |
  \hat{\br}_{i,\text{odd}})^T \in (\hat{\bR}_{+,\text{even}} |
  \hat{\bR}_{+,\text{odd}})^T,
  \]
  there exists a column vector
  \[
  \hat{\br}_j = (-\hat{\br}_{i,\text{even}} |
  \hat{\br}_{i,\text{odd}})^T \in (\hat{\bR}_{-,\text{even}} |
  \hat{\bR}_{-,\text{odd}})^T.
  \]
  Then
  \[
    \hat{\br}_i - \hat{\br}_j = 2(\hat{\br}_{i,\text{even}} |
    \boldsymbol{0})^T.
  \]
  The set of vectors $\hat{\br}_i - \hat{\br}_j$ are linearly
  independent since $\hat{\br}_i, \hat{\br}_j$ are eigenvectors of
  $\hat{\bR}$, then 
  columns of $\hat{\bR}_{+,\text{even}}$ are linearly independent.
  Thus $\hat{\bR}_{+,\text{even}}$ is invertible. 
\end{proof}

\subsubsection{Illustrative examples: $M \leq 5$}
We examine the cases for small $M$ to find out the formal solution for
generic $M$. The simplest system is the case for $M = 3$.  The
variables are $V = (\bar{u}_1, \bar{\sigma}_{12},
\bar{f}_{e_1+2e_2})^T$, and matrices $\bM$ and $\bQ$ in
\eqref{eq:simple-velocity} are 
\[
  \bM =\left(
  \begin{array}{ccc}
    0 & 1 & 0 \\
    1 & 0 & 2\\
    0 & 1 & 0
  \end{array} \right),
  \quad
  \bQ = \left(
  \begin{array}{ccc}
    0 &   &  \\
      & 1 &  \\
      &   & 1
  \end{array} \right).
\]
Since $\bar{\sigma}_{12}$ is constant and the velocity is
\[
    \bar{u}_1 = -\bar{\sigma}_{12}\frac{\bar{y}}{\Kn} + c_0,
\]
it is clear that the solution of $\bar{u}_1$ is not able to capture
the boundary layer since here $\bar{u}_1$ is a linear function of
$\bar{y}$. To capture the boundary layer of velocity, we need more
moments thus we turn to the case $M = 4$. For $M = 4$, the equations
\eqref{eq:velocity} are
\begin{equation}\label{eq:vel-M4}
  \begin{aligned}
    &\od{\bar{\sigma}_{12}}{\bar{y}} = 0, \\
    &\od{\bar{u}_1}{\bar{y}} + 2\od{\bar{f}_{e_1+2e_2}}{\bar{y}}
    = -\frac{1}{\Kn}\bar{\sigma}_{12}, \\
    &\od{\bar{\sigma}_{12}}{\bar{y}} +
    3\od{\bar{f}_{e_1+3e_2}}{\bar{y}} = -\frac{1}{\Kn}\bar{f}_{e_1+2e_2}, \\
    &\od{\bar{f}_{e_1+2e_2}}{\bar{y}} = -\frac{1}{\Kn}\bar{f}_{e_1+3e_2}. \\
  \end{aligned}
\end{equation}
The solution gives us the expression of velocity as
\begin{equation} \label{eq:velo-ori}
  \bar{u}_1 = - \bar{\sigma}_{12}\frac{\bar{y}}{\Kn} - 2 \bar{f}_{e_1+2e_2} + c_0
\end{equation}
from the second equation in \eqref{eq:vel-M4}. Here we need to solve the equations
\eqref{eq:matrix-hatM} for $\hat{V} = (\bar{f}_{e_1+2e_2}, \bar{f}_{e_1+3e_2})^T$,
where
\begin{equation}\label{mat:tridiagonal}
  \hat{\bM} = \left(
  \begin{array}{cc}
    0 & 3\\
    1 & 0
  \end{array}
  \right ).
\end{equation}
The matrix $\hat{\bM}$ can be decomposited as
$\hat{\bM} = \hat{\bR} \hat{\bLambda} \hat{\bR}^{-1}$,
\[
    \hat{\bLambda} = \left( 
    \begin{array}{cc}
        \sqrt{3} &  \\
                 & -\sqrt{3}
    \end{array} \right), \qquad
    \hat{\bR} = \left(
    \begin{array}{cc}
        1                  & 1 \\
        \frac{1}{\sqrt{3}} & -\frac{1}{\sqrt{3}}
    \end{array} \right).
\]
Hence, the solution of system \eqref{eq:matrix-hatM} is
\begin{equation*}
  \hat{V} = \hat{\bR} \exp \left(-\frac{\bar{y}}{\Kn} \hat{\bLambda}^{-1} \right)
  \hat{\bR}^{-1} \hat{V}^{(0)}.
\end{equation*}
By setting
$\hat{\bc}= (\hat{c}_1, \hat{c}_2)^T = \hat{\bR}^{-1}\hat{V}^{(0)}$,
the equations above result in
\begin{equation*}
    \hat{V}=\left(\begin{array}{l}
        \bar{f}_{e_1+2e_2}\\
        \bar{f}_{e_1+3e_2}
    \end{array}\right)
    =\hat{\bR} \exp \left(-\frac{\bar{y}}{\Kn} \hat{\bLambda}^{-1}
    \right)\hat{\bc}
    =
    \left( \begin{array}{l}
        \hat{c}_1 \exp(-\frac{\bar{y}}{\sqrt{3}\Kn}) +
        \hat{c}_2\exp(\frac{\bar{y}}{\sqrt{3}\Kn})\\ 
        \frac{\sqrt{3}}{3} \hat{c}_1
        \exp(-\frac{\bar{y}}{\sqrt{3}\Kn}) - \frac{\sqrt{3}}{3} \hat{c}_2 
        \exp(\frac{\bar{y}}{\sqrt{3}\Kn})
    \end{array} \right).
\end{equation*}
The exponential terms provide us the boundary layer. Since all the
variables have to remain finite as $\bar{y} \to \infty$, the term
$\exp(\frac{1}{\sqrt{3}\Kn}\bar{y})$ has to be dropped. Therefore,
\[
   \left(\begin{array}{l}
        \bar{f}_{e_1+2e_2}\\
        \bar{f}_{e_1+3e_2}
    \end{array}\right)
    =\hat{\bR} \left(
      \begin{array}{cc}
        \exp \left(-\frac{\bar{y}}{\Kn} \hat{\bLambda}_+^{-1} \right) 
        & \\ & \bzero 
      \end{array}
    \right) \hat{\bc}
    = \hat{c}_1 \exp(-\frac{\bar{y}}{\sqrt{3}\Kn}) \left( 
      \begin{array}{c}
        1 \\ 
        \frac{\sqrt{3}}{3}
      \end{array} 
    \right).
\]
Here $\hat{\bLambda}_{+}$ is a $1 \times 1$ matrix with its entry as
$\sqrt{3}$.  Applying the linearized boundary condition
\eqref{bc:linear}, i.e.
\begin{equation*}
  \begin{aligned}
    &- \frac{1}{\sqrt{2\pi}} \bar{u}_1 = S(1,1) \bar{\sigma}_{12} + S(1,2)
    \bar{f}_{e_1+2e_2} + S(1,3)\bar{f}_{e_1+3e_2}, \\
    &- \frac{2}{\sqrt{2\pi}} \bar{u}_1 = S(3,1) \bar{\sigma}_{12} + S(3,2)
    \bar{f}_{e_1+2e_2} + S(3,3)\bar{f}_{e_1+3e_2},
  \end{aligned}
\end{equation*}
we can obtain 
\begin{equation*}
\hat{c}_1 = -\frac{\sqrt{\pi}(\chi - 2)}{2(\sqrt{3\pi}(2 -
    \chi) + 2\sqrt{2}\chi)} \bar{\sigma}_{12}, \quad c_0 =
  \sqrt{\frac{\pi}{2}} \frac{\chi - 2}{\chi}\left( 1 +
  \frac{\sqrt{2}\chi}{4\sqrt{2}\chi + 2\sqrt{3\pi}(2 - \chi)}\right)
  \bar{\sigma}_{12}.  
\end{equation*}
Then the solution of velocity is 
\begin{equation*}
  \bar{u}_1 = - \bar{\sigma}_{12}\frac{\bar{y}}{\Kn}
  -2\hat{c}_1\exp(-\frac{\bar{y}}{\sqrt{3}\Kn}) + c_0.
\end{equation*}

Similar procedure can be carried out for greater $M$. For example, if
we set $M = 5$, then
$V = (\bar{u}_1, \bar{\sigma}_{12}, \bar{f}_{e_1+2e_2},
\bar{f}_{e_1+3e_2}, \bar{f}_{e_1+4e_2})^T$ 
and $\hat{V} = (\bar{f}_{e_1+2e_2}, \bar{f}_{e_1+3e_2}, \bar{f}_{e_1+4e_2})^T$. The
matrix $\hat{\bM} = \hat{\bR} \hat{\bLambda}\hat{\bR}^{-1}$ in
\eqref{eq:matrix-hatM} is
\begin{equation*}
    \hat{\bM} = \left(
    \begin{array}{ccc}
        0 & 3 & 0 \\
        1 & 0 & 4 \\
        0 & 1 & 0
    \end{array}
    \right ) \text{ with } 
    \hat{\bLambda} = \left(
    \begin{array}{ccc}
        \sqrt{7} &   &  \\
                 & 0 &  \\
                 &   & -\sqrt{7}
    \end{array} \right),
  ~~
    \hat{\bR} = \left(
    \begin{array}{ccc}
        1          & 1    & 1           \\
        \sqrt{7}/3 & 0    & -\sqrt{7}/3 \\
        1/3        & -1/4 & 1/3
    \end{array} \right).
\end{equation*}
Notice that $M = 5$ is odd, so zero is a simple eigenvalue of
$\hat{\bM}$ . This vanished eigenvalue provides a constant factor in
the exponential terms in the boundary layer, while the eigenvalue
$\sqrt{7}$ of matrix $\hat{\bM}$ provides the only stable term which
survives in the solution. The solution of \eqref{eq:matrix-hatM} is
\begin{equation}
    \hat{V}=
    \left(\begin{array}{l}
        \bar{f}_{e_1+2e_2}\\
        \bar{f}_{e_1+3e_2}\\
        \bar{f}_{e_1+4e_2}\\
    \end{array}\right)
    =\hat{\bR} \left(
      \begin{array}{cc}
        \exp \left(-\frac{\bar{y}}{\Kn} \hat{\bLambda}_+^{-1} \right) 
        & \\ & \bzero 
      \end{array}
    \right) \hat{\bc}
    = \hat{c}_1 \exp(-\frac{\bar{y}}{\sqrt{7}\Kn})
    \left( \begin{array}{c}
        1 \\ 
        \frac{\sqrt{7}}{3} \\
        \frac{1}{3}
    \end{array} \right),
\end{equation}
where
$\hat{\bc}=(\hat{c}_1, \hat{c}_2, \hat{c}_3)^T =
\hat{\bR}^{-1}\hat{V}^{(0)}$
and the entry of the $1 \times 1$ matrix $\hat{\bLambda}_+$ is
$\sqrt{7}$.  Similar as the case $M = 4$, there are $2$ coefficients
$c_0$ and $\hat{c}_1$ to be determined.  To fix the coefficients, we
utilize two boundary conditions by setting $\beta_2 = 1,3$ in
\eqref{bc:linear}
\begin{equation*}
  \begin{aligned}
    &- \frac{1}{\sqrt{2\pi}} \bar{u}_1 = S(1,1) \bar{\sigma}_{12} + S(1,2)
    \bar{f}_{e_1+2e_2} + S(1,3)\bar{f}_{e_1+3e_2} + S(1,4)\bar{f}_{e_1+4e_2} , \\
    &- \frac{2}{\sqrt{2\pi}} \bar{u}_1 = S(3,1) \bar{\sigma}_{12} + S(3,2)
    \bar{f}_{e_1+2e_2} + S(3,3)\bar{f}_{e_1+3e_2} + S(3,4)\bar{f}_{e_1+4e_2}.
  \end{aligned}
\end{equation*}
Direct calculations yield
\begin{equation*}
    \hat{c}_1 = -\frac{3\sqrt{\pi}(\chi - 2)}{2(3\sqrt{7\pi}(\chi -
    2) - 10\sqrt{2}\chi)} \bar{\sigma}_{12}, \quad c_0 = \sqrt{\frac{\pi}{2}}
    \frac{\chi - 2}{\chi}\left(1 -
    \frac{2\sqrt{2}\chi}{3\sqrt{7\pi}(\chi - 2) - 10\sqrt{2}\chi}\right)
    \bar{\sigma}_{12}. 
\end{equation*}
Then the solution of velocity is given by
\begin{equation*}
    \bar{u}_1 = -\bar{\sigma}_{12}\frac{\bar{y}}{\Kn} - 2\hat{c}_1
    \exp\left(-\frac{\bar{y}}{\sqrt{7}\Kn}\right) + c_0.
\end{equation*}

\subsubsection{General case: arbitrary $M$}
Now we are ready to present the formal solution for arbitrary
$M$. Following the examples above, we have to drop those unbounded
factors to attain a stable solution that only the terms contributed
from the positive eigenvalues of $\hat{\bM}$ are kept. Thus the stable
solution of \eqref{eq:matrix-hatM} is
\begin{equation}\label{sol:general}
  \hat{V}(\bar{y}) = \hat{\bR} \left(
    \begin{array}{cc}
      \exp \left(-\frac{\bar{y}}{\Kn} \hat{\bLambda}_{+}^{-1}\right) 
      & \\ & \bzero
    \end{array}
  \right)
  \hat{\bc},
\end{equation}
where
$\hat{\bc}=(\hat{c}_1, \cdots, \hat{c}_{M-2})^T = \hat{\bR}^{-1}
\hat{V}^{(0)}$.
Clearly, there are only the beginning $\lfloor \dfrac{M}{2} \rfloor - 1$
entries in $\hat{\bc}$ appears in $\hat{V}(\bar{y})$. With the expression of
$\bar{f}_{e_1+2e_2}(\bar{y})$ provided as the first entry of
$\hat{V}(\bar{y})$, the velocity is again given by the second equation
in \eqref{eq:velocity} as
\begin{equation}\label{sol:velocity}
  \bar{u}_1(\bar{y}) = - \bar{\sigma}_{12}\frac{\bar{y}}{\Kn} - 2
  \be_1^T\hat{V}(\bar{y}) + c_0, 
\end{equation}
where $\be_1 = (1, 0, \cdots, 0)^T$.  Since $c_0$ in the expression of
$\bar{u}_1(\bar{y})$ is also to be determined, there are in total
$\lfloor \dfrac{M}{2} \rfloor$ indeterminate coefficients in $V(\bar{y})$.

Combining \eqref{sol:general} and \eqref{sol:velocity} with linearized
boundary condition \eqref{bc:linear}, we can obtain the boundary
condition for \eqref{eq:matrix-hatM} as
\begin{equation}\label{eq:simple-bc}
    -\frac{(\beta_2 -1)!!}{\sqrt{2\pi}} \bar{u}_1 = S(\beta_2,1)\bar{\sigma}_{12} + 
    \sum_{\alpha_2=2}^{M-1} S(\beta_2,\alpha_2) \bar{f}_{e_1+\alpha_2e_2},
    \quad \beta_2=1,3,\cdots,2\lfloor\frac{M}{2}\rfloor-1.
\end{equation}
The total number of boundary condition is $\lfloor\frac{M}{2}\rfloor$,
which may fix all coefficients in the solution of $V(\bar{y})$.  Once these
coefficients are fixed by the boundary conditions, we eventually
attain $\bar{u}_1$ formated as
\begin{equation}\label{eq:sol_u1_formal}
  \begin{aligned}
    \bar{u}_1(\bar{y}) &= - \bar{\sigma}_{12}\frac{\bar{y}}{\Kn} - 2
    \be_1^T\hat{V}(\bar{y}) + c_0 \\ 
    &= - \bar{\sigma}_{12}\frac{\bar{y}}{\Kn} - 2 \be_1^T \hat{\bR} \left(
      \begin{array}{cc}
        \exp \left(-\frac{\bar{y}}{\Kn} \hat{\bLambda}_{+}^{-1}\right) 
        & \\ & \bzero
      \end{array}
    \right)
    \hat{\bc} + c_0\\
    &= - \bar{\sigma}_{12}\frac{\bar{y}}{\Kn} - 2\sum_{i = 1}^{\lfloor
      \frac{M-2}{2} \rfloor}\hat{c_i}
    \exp\left(-\frac{\bar{y}}{\Kn\hat{\lambda}_i}\right) +c_0.
    \end{aligned}
\end{equation}
We let $\bar{y} = 0$ in \eqref{eq:sol_u1_formal} to have
$\bar{u}_1 = -2 \sum_{i = 1}^{\lfloor \frac{M}{2} \rfloor -1}\hat{c_i} +
c_0$
and substitute it into \eqref{eq:simple-bc} to obtain the following
linear system
\begin{equation}\label{sol:linear}
\begin{aligned}
  & \quad \qquad - \frac{-2 \displaystyle \sum_{i = 1}^{\lfloor
      \frac{M}{2} \rfloor-1}\hat{c_i} + c_0}{\sqrt{2\pi}}\left(
  \begin{array}{c}
    1 \\
    (3 - 1)!!\\
    \vdots\\
    (2\lfloor \frac{M}{2} \rfloor - 2)!!
  \end{array}\right) = \bar{\sigma}_{12} \left(
  \begin{array}{c}
    S(1,1) \\ (S(3,1) \\ \vdots \\ 
    S(2\lfloor \frac{M}{2} \rfloor -1,1) 
  \end{array}
\right) \\
& + \left( \begin{array}{cccc}
  S(1,2) & S(1,3) &\cdots & S(1,M-1) \\
  S(3,2) & S(3,3) &\cdots & S(3,M-1) \\
  \vdots & \vdots &\ddots & \vdots \\
  S(2\lfloor \frac{M}{2} \rfloor -1,2) 
         & S(2\lfloor \frac{M}{2} \rfloor -1,3) 
                  & \cdots & S(2\lfloor \frac{M}{2} \rfloor -1,M-1)
\end{array} \right) 
\hat{\bR} \left(
    \begin{array}{c}
      \hat{c}_1 \\ \vdots \\ \hat{c}_{\lfloor \frac{M}{2} \rfloor - 1}
      \\ 0 \\ \vdots \\ 0
    \end{array}
  \right)
\end{aligned}
\end{equation}
Clearly, this is a linear system of for
$\bc = (c_0, \hat{c}_1, \cdots, \hat{c}_{\lfloor \frac{M}{2}
  \rfloor-1})^T$. Precisely, we let 
\[
\bh = \frac{1}{\sqrt{2\pi}} \left(
      \begin{array}{c}
        1 \\
        (3 - 1)!!\\
        \vdots\\
        (2\lfloor \frac{M}{2} \rfloor - 2)!!
      \end{array}\right), 
\qquad \bb = \left(
    \begin{array}{c}
      S(1,1)\\
      S(3,1)\\
      \vdots\\
      S(2\lfloor \frac{M}{2} \rfloor-1,1)
    \end{array} \right),
\]
\[
\bS = \left( 
      \begin{array}{cccc}
        S(1,2) & S(1,3) & \cdots & S(1,M-1) \\
        S(3,2) & S(3,3) & \cdots & S(3,M-1) \\
        \vdots & \vdots & \ddots & \vdots   \\
        S(2\lfloor \frac{M}{2} \rfloor-1 ,2) 
               & S(2\lfloor \frac{M}{2} \rfloor-1, 3) 
                        & \cdots & S(2\lfloor \frac{M}{2} \rfloor-1, M-1)
      \end{array} \right),
\]
then the system \eqref{sol:linear} is formulated as
\begin{equation}\label{sol:linear1}
\bA \bc = -\sigma_{12} \bb,
\end{equation}
where $\bA = \left( \bh, (\bS - 2 \bh \be_1^T)\hat{\bR}_+ \right)$.

To fix the parameters in $\bc$, the unique solvability of
\eqref{sol:linear1} is required. Currently, we can only claim the
system \eqref{sol:linear1} is uniquely solvable when $\chi$ is an
algebraic number. Precisely, we have the following theorem:
\begin{theorem}
  $\left| \bA \right| \neq 0$ if $\chi$ is an algebraic number.
\end{theorem}
\begin{proof}
  We times $\bA$ by
  $\left(
    \begin{array}{cc} 
      1 & 2 \be_1^T \hat{\bR}_+ \\ 
      \bzero & \bI \end{array} 
  \right)$ to obtain $(\bh, \bS \hat{\bR}_+)$. Thus $|\bA| = |(\bh,\bS
  \hat{\bR}_+)|$. 

  We retrieve the coefficient $h_m(\chi)$ in $S(k,m)$ to have
  \[
  \bS = \bS^\star_0 \boldsymbol{H} 
  \]
  where
  $\boldsymbol{H} = \dfrac{1}{\sqrt{2\pi}} \diag \{ h_2(\chi),
  h_3(\chi), \cdots, h_{M-1}(\chi) \}$ and
  \[
  \bS^\star_0 = \left( 
    \begin{array}{cccc}
      S^\star(1,2) & S^\star(1,3) & \cdots & S^\star(1,M-1) \\
      S^\star(3,2) & S^\star(3,3) & \cdots & S^\star(3,M-1) \\
      \vdots & \vdots & \ddots & \vdots   \\
      S^\star(2\lfloor \frac{M}{2} \rfloor-1 ,2) 
                   & S^\star(2\lfloor \frac{M}{2} \rfloor-1, 3) 
                                  & \cdots & S^\star(2\lfloor
      \frac{M}{2} \rfloor-1, M-1) 
    \end{array} \right).
  \]

  By the recursion relation \eqref{eq:rec} of $S^\star(k,m)$, we have
  that
  \[
  \boldsymbol{L} \bS^\star_0 = \left( 
    \begin{array}{c}
      S^\star(1,2), ~ S^\star(1,3), ~ \cdots, ~ S^\star(1,M-1) \\
      \bS^\star_1
    \end{array} 
  \right),
  \]
  where 
  \[
  \boldsymbol{L} = \left(
    \begin{array}{ccccc}
      1 &&&& \\
      -2 & 1 &&& \\
        & -4 & 1 && \\
        & & \ddots & \ddots & \\
        & & & -(2\lfloor \frac{M}{2} \rfloor-2) & 1
    \end{array}
  \right),
  \]
  and
  \[
  \bS^\star_1 = \left( 
    \begin{array}{cccc}
      2S^\star(2,1) & 3S^\star(2,2) & \cdots & (M-1)S^\star(2,M-2) \\
      2S^\star(4,1) & 3S^\star(4,2) & \cdots & (M-1)S^\star(4,M-2) \\
      \vdots & \vdots & \ddots & \vdots   \\
      2S^\star(2\lfloor \frac{M}{2} \rfloor-2,1) 
                    & 3S^\star(2\lfloor \frac{M}{2} \rfloor-2, 2) 
                                    & \cdots & (M-1)S^\star(2\lfloor
      \frac{M}{2} \rfloor-2, M-2) 
    \end{array} 
  \right).
  \]

  Noticing that $\boldsymbol{L} \bh = (1/\sqrt{2\pi}, 0, \cdots, 0)^T$, we have
  that
  \[
  (\bh, \bS \hat{\bR}_+) = \boldsymbol{L}^{-1} \left(
    \begin{array}{cc}
      1/\sqrt{2\pi} & (S^\star (1, 2), \cdots, S^\star (1, M-1)) \\
      \bzero & \bS^\star_1
    \end{array}
  \right) \left(
    \begin{array}{cc} 1 & \\ & \boldsymbol{H} \hat{\bR}_+ \end{array}
  \right).
  \]
  Thus we need only to verify the determinant of
  $\bS^\star_1 \boldsymbol{H} \hat{\bR}_+$ is not vanished. 
  Consider the permutation matrix in \eqref{mat:P_sigma},
  we then see that
  \[
  \bS^\star_1 \boldsymbol{P}_\sigma = (\bS^\star_{\text{even}},
  \bS^\star_{\text{odd}}),
  \]
  where $\bS^\star_{\text{even}}$ is made with the even columns of
  $\bS^\star_1$ as
  \[
  \bS^\star_{\text{even}} = \left(
    \begin{array}{cccc}
      3S^\star(2,2) & 5S^\star(2,4) & 7S^\star(2,6) & \hdots \\
      3S^\star(4,2) & 5S^\star(4,4) & 7S^\star(4,6) & \hdots \\
      \vdots & \vdots & \vdots & \vdots \\
      3S^\star(2\lfloor \frac{M}{2} \rfloor-2,2) & 5S^\star(2\lfloor
      \frac{M}{2} \rfloor-2,4) & 7S^\star(2\lfloor \frac{M}{2}
      \rfloor-2,6)& \hdots 
    \end{array}
  \right),
  \]
  and $\bS^\star_{\text{odd}}$ is made with the odd columns of
  $\bS^\star_1$ as
  \[
  \bS^\star_{\text{odd}} = \left(
    \begin{array}{cccc}
      2S^\star(2,1) & 4S^\star(2,3) & 6S^\star(2,5) &\hdots \\
      2S^\star(4,1) & 4S^\star(4,3) & 6S^\star(2,5) &\hdots \\
      \vdots & \vdots & \vdots & \vdots \\
      2S^\star(2\lfloor \frac{M}{2} \rfloor-2,1) & 4S^\star(2\lfloor
      \frac{M}{2} \rfloor-2,3) & 6S^\star(2\lfloor \frac{M}{2}
      \rfloor-2,5) & \hdots  
    \end{array}
  \right).
  \]
  With the integral properties of $S^{\star}(k,m)$ in ~\ref{sec:bc},
  $\bS^{\star}_{\text{even}}$ is a lower triangular matrix and each
  entry in its lower triangular part is an
  algebraic number$\times \sqrt{2\pi}$.

  The diagonal matrix $\boldsymbol{H}$ is turned into
  \[
  \boldsymbol{P}_\sigma^{-1} \boldsymbol{H} \boldsymbol{P}_\sigma =
  \dfrac{1}{\sqrt{2\pi}}\left(
    \begin{array}{cc}
      \bI & \\
          & \dfrac{2-\chi}{\chi} \bI
    \end{array}
  \right).
  \]
  We then have that
  \[
  \begin{aligned}
    \bS^\star_1 \boldsymbol{H} \hat{\bR}_+ = &
    \bS^\star_1 \boldsymbol{P}_\sigma ~ \boldsymbol{P}_\sigma^{-1}
    \boldsymbol{H} \boldsymbol{P}_\sigma ~ \boldsymbol{P}_\sigma^{-1}
    \hat{\bR}_+ \\
    = & \dfrac{1}{\sqrt{2\pi}} 
    (\bS^\star_{\text{even}}, \bS^\star_{\text{odd}})
    \left(
      \begin{array}{cc}
        \bI & \\
            & \dfrac{2-\chi}{\chi} \bI
      \end{array}
    \right) \left(
      \begin{array}{c}
        \hat{\bR}_{+,\text{even}} \\
        \hat{\bR}_{+,\text{odd}}
      \end{array}
    \right) \\
    = & \dfrac{1}{\sqrt{2\pi}} \left( \bS^\star_{\text{even}}
      \hat{\bR}_{+,\text{even}} + \dfrac{2-\chi}{\chi}
      \bS^\star_{\text{odd}} \hat{\bR}_{+,\text{odd}} \right).
  \end{aligned}
  \]
  Since $\hat{\bR}_{+,\text{even}}$ is invertible by Lemma
  \ref{lem:R_peven_invertible},
  \begin{equation}\label{eq:SHR}
    \bS^\star_1 \boldsymbol{H} \hat{\bR}_+ = \dfrac{1}{\sqrt{2\pi}}
    \left( \bS^\star_{\text{even}} + \dfrac{2-\chi}{\chi}
      \bS^\star_{\text{odd}} \hat{\bR}_{+,\text{odd}}
      \hat{\bR}_{+,\text{even}}^{-1}\right)\hat{\bR}_{+,\text{even}},
  \end{equation}
  and we only need to verify the matrix
  $\bS^\star_{\text{even}} + \dfrac{2-\chi}{\chi}
  \bS^\star_{\text{odd}} \hat{\bR}_{+,\text{odd}}
  \hat{\bR}_{+,\text{even}}^{-1}$
  in \eqref{eq:SHR} is not singular. Considering the polynomial of
  $\lambda$ defined by
  \[
  p(\lambda) \triangleq
  \left|\dfrac{\lambda}{\sqrt{2\pi}}\bS^\star_{\text{even}} +
    \dfrac{2-\chi}{\chi} \bS^\star_{\text{odd}}
    \hat{\bR}_{+,\text{odd}} \hat{\bR}_{+,\text{even}}^{-1}\right|,
  \] 
  we point out that $p(\lambda)$ is a polynomial with all coefficients
  to be algebraic numbers, since $\chi$ is assumed to be algebraic,
  and entries of matrices
  $\dfrac{1}{\sqrt{2\pi}}\bS^\star_{\text{even}}$,
  $\hat{\bR}_{+,\text{even}}^{-1}$, $\bS^\star_{\text{odd}}$, and
  $ \hat{\bR}_{+,\text{odd}}$ are all algebraic
  numbers. Particularly, the coefficient of the leading term of
  $p(\lambda)$ is the product of all diagonal entries in matrix
  $\dfrac{1}{\sqrt{2\pi}}\bS^{\star}_{\text{even}}$ and thus is not
  vanished. Therefore, $p(\lambda) = 0$ can only valid for $\lambda$
  to be an algebraic number, too. Thus $p(\sqrt{2\pi}) \neq 0$ and
  consequently $\left| \bA \right| \neq 0$. We conclude that the
  linear system \eqref{sol:linear1} is uniquely solvable.
\end{proof}
\begin{remark}
  Definitely, we speculate that $\left| \bA \right| \neq 0$ for all
  $\chi$, while currently we can not prove it unfortunately. If we
  take $\left| \bA \right|$ as a function of $\chi$, it is clearly a
  continuous function. The theorem above declares that the value of
  the function is not zero on all algebraic numbers, and by the
  continuity of the function, the roots for $\left| \bA \right| = 0$
  can not be dense on $\mathbb{R}$ at least.
\end{remark}

\subsection{Convergence in moment order}\label{sec:convergence}
Let us reveal below the connection of the linearized HME and the
linearized Boltzmann equation in \cite{Williams2001}. For Kramers'
problem, the boundary condition we proposed is related with that in
\cite{Williams2001}, either. Roughly speaking, our system is
illustrated to be a particular discretization of the equation in
\cite{Williams2001}. This allows to examine the convergence of the
solution of our systems to the numerical results of the equation in
\cite{Williams2001}. It is demonstrated that the solution converges to
that of the linearized Boltzmann equation in \cite{Williams2001} with
the increasing of moment order. Let us start from a brief review of
the main result on the Kramers' problem in \cite{Williams2001}.

For the time independent Boltzmann equation
\begin{equation}\label{eq:space-Boltz}
\bxi\cdot\nabla_{\bx} f = Q(f,f),
\end{equation}
considering here only Kramers' problem is studied, we linearize the
distribution function $f$ as 
\begin{equation}\label{eq:linear-func}
    f(\bx,\bxi) = \mathcal{M}(\bx,\bxi) [ 1 + h(\bx,\bxi) ],
\end{equation}
where $h(\bx,\bxi)$ is a disturbance term caused by the small
perturbation near the local equilibrium Maxwellian
$\mathcal{M}(\bx,\bxi)$, which has the form 
\[
    \mathcal{M}(\bx,\bxi) =
    \frac{\rho_{0}(x,y)}{(2\pi\theta_{0}(x,y))^{3/2}}
    \exp\left( -\frac{(\xi_x - u(y))^2 + \xi_y^2 + \xi_z^2}
    {2\theta_{0}(x,y)} \right).
\]
Here $\rho_0$ and $\theta_0$ are same as that in
\eqref{eq:dimensionless}, and 
\[
    u(y) = K y.
\]
Inserting \eqref{eq:linear-func} into the time independent Boltzmann
equation \eqref{eq:space-Boltz}, and discarding the high-order small
quantities, we can obtain
\begin{equation}\label{eq:insert-Boltz}
    \bxi \cdot \nabla_{\bx} \mathcal{M} + \mathcal{M} \bxi \cdot \nabla_{\bx} h =
    J(\mathcal{M},h),
\end{equation}
where $J(\mathcal{M},h)$ is the linearized BGK collision 
\[
    J(\mathcal{M},h) = - \mathcal{M}\left\{ \nu h - \frac{\nu}{\sqrt{2\pi\theta_0}^3}
    \int h(y, \bxi') \left[1 + \frac{1}{\theta_0} \bxi \cdot \bxi' +
        \frac{2}{3}\left(\frac{|\bxi|^2}{2\theta_0}- \frac{3}{2}\right)
        \left(\frac{|\bxi'|^2}{2\theta_0}-\frac{3}{2}\right)\right]
        \exp\left(-\frac{|\bxi|^2}{2\theta_0}\right) \dd \bxi' \right\},
\]
where $\nu$ is the collision frequency of BGK model.
For convenience, we introduce the dimensionless variables
\[
    \xi_i = \sqrt{\theta_0}\bar{\xi}_i, \qquad K = \sqrt{\theta_0}K_0,
    \qquad \bx=L\bar{\bx},\qquad 
    \Kn = \frac{\sqrt{\theta_0}}{L\nu}.
\]
Then direct calculations and some simplification yield
\[
    \begin{aligned}
        &\bar{\xi}_x
        \bar{\xi}_y K_0 +  \bar{\xi}_y \pd{h(\bar{y},\bar{\bxi})}{\bar{y}}  \\
        &= - \frac{1}{\Kn} \left\{ h(\bar{y}, \bar{\bxi}) - \frac{1}{\sqrt{2\pi}^3} \int
        h(\bar{y}, \bar{\bxi'}) \left[1 + \bar{\bxi} \cdot \bar{\bxi'} + 
            \frac{2}{3}\left(\frac{|\bar{\bxi}|^2-3}{2}\right)
            \left(\frac{|\bar{\bxi'}|^2-3}{2}\right)\right]
            \exp(-\frac{|\bar{\bxi}|^2}{2}) \dd \bar{\bxi}'\right\},
    \end{aligned}
\]
and
\begin{equation}\label{eq:integralu}
 \bar{u}_1(\bar{y}) = K_0 \bar{y} + \frac{1}{\sqrt{2\pi}}
            \int_{-\infty}^{\infty} Z(\bar{y},\bar{\xi}) \exp\left(
            -\frac{\bar{\xi}^2}{2} \right) \dd \bar{\xi},
\end{equation}
where 
\[
    Z(\bar{y}, \bar{\xi}_y) = \frac{1}{2\pi} \int_{-\infty}^{\infty}
    \int_{-\infty}^{\infty} \bar{\xi}_x h(\bar{y}, \bar{\bxi})
    \exp\left(-\frac{\bar{\xi}_x^2+\bar{\xi}_z^2}{2}\right) \dd
    \bar{\xi}_x \dd \bar{\xi}_z .
\]
Then we have
\begin{equation}\label{eq:William}
  K_0 \bar{\xi} + \bar{\xi} \pd{Z(\bar{y},\bar{\xi})}{\bar{y}}
  =\frac{1}{\Kn} \left( 
    - Z(\bar{y},\bar{\xi}) + \frac{1}{\sqrt{2\pi}} \int_{\bbR}
    Z(\bar{y}, \bar{\xi'}) \exp\left (-\frac{\bar{\xi'}^2}{2}\right)
    \dd \bar{\xi'} \right).
\end{equation}
From \eqref{eq:sol_u1_formal} and \eqref{eq:integralu} we notice that
\[
K_0 = - \frac{\bar{\sigma}_{12}}{\Kn},
\]
and
\begin{equation}\label{eq:integralZ}
  \begin{aligned}
    \bar{f}_{e_1+ie_2} &= \frac{1}{i!}\frac{1}{\sqrt{2\pi}^3}
    \int_{\bbR^3}\He_1(\bar{\xi}_x) \He_i(\bar{\xi}_y)
    h(\bar{y},\bar{\bxi}) \exp\left(-\frac{\bar{\bxi}^2}{2}\right)
    \dd\bar{\bxi}, \quad i=1, \dots,M-1,\\
    &= \frac{1}{i!}\frac{1}{\sqrt{2\pi}}
    \int_{\bbR} \He_i(\bar{\xi}) Z(\bar{y},\bar{\xi})
    \exp\left(-\frac{\bar{\xi}^2}{2}\right) 
    \dd\bar{\xi}, \quad i= 1,\dots,M-1.
    \end{aligned}
\end{equation}
Following \cite{Williams2001}, we use the model which is
also based on the diffuse-specular process, and boundary condition can
be written as 
\[
f(0, \bxi) = \chi N f^W_M(\bx, \bxi) + (1-\chi) f(0, \bxi^{\ast}),
\]
where $f^W_M, \bxi^{\ast}$ is defined in \eqref{eq:equilibrium}, $N$
is a normalizing factor to be determined
\cite{Williams2001}. Using the zero mass flux condition
\[
\int_{\xi_y<0} \xi_y f(0,\bxi) \dd\bxi + \int_{\xi_y>0} \xi_y
f(0,\bxi) \dd\bxi = 0
\]
at $y=0$ to calculate the $N$.
After the linearization and nondimensionalization, the boundary
condition in Cartesian velocity coordinates as follows
\[
\begin{aligned}
h(0, \bar{\xi}_x, \bar{\xi}_y, \bar{\xi}_z) &= \chi[ \bar{\xi}_x
  \bar{u}_1^W + \delta(\frac{\bar{\xi}^2}{2} - 2)] + (1- \chi) h(0, \bar{\xi}_x,
-\bar{\xi}_y, \bar{\xi}_z) \\
&- \frac{\chi}{2\pi} \int^0_{-\infty} \bar{\xi}'_y \dd \bar{\xi}'_y
\int_{\bbR^2} h(0, \bar{\xi}'_x, \bar{\xi}'_y, \bar{\xi}'_z)
\exp(-\frac{\bar{\bxi}'^2}{2}) \dd \bar{\xi}'_x \bar{\xi}'_z; \quad
\bar{\xi}_y > 0. 
\end{aligned}
\]

Consider the Kramers' problem with boundary condition $\bar{u}_1^W =
0$ and $\delta = (\theta_W - \theta_0)/ \theta_W = 0$, then we have
\begin{equation}\label{eq:William-bc}
  Z(0, \bar{\xi}) = (1 - \chi) Z(0, -\bar{\xi}) ; \quad \bar{\xi} > 0. 
\end{equation}

The equation \eqref{eq:William} is an integral equation on $\bar{\xi}$
and differential equation on $y$. Here we discretize it on
$\bar{\xi}$.  Consider the Gauss-Hermite quadrature with $M\in\bbN$
points, and denote the weights and integral points by $\omega_i$ and
$\bar{\xi}_i$, $i=1,\cdots,M$. If we sort the $\bar{\xi}_i$ in
decending order, then $\bar{\xi}_i = \lambda_i$ in
\eqref{mat:pos-neg}. Let $Z(\bar{y})^k =
Z(\bar{y},\bar{\xi}_k)$ and
$\bZ(\bar{y})=(Z(\bar{y})^1,\cdots,Z(\bar{y})^M)^T$ and
$\bomega=(\omega_1,\cdots,\omega_M)^T$, then we have
\begin{equation}
    K_0 \bLambda \boldsymbol{1} +
    \bLambda\od{\bZ(\bar{y})}{\bar{y}}=\frac{1}{\Kn}\left( 
    \boldsymbol{1}\bomega^T-\bI \right)\bZ(\bar{y}),
\end{equation}
where $\bLambda$ is same as the \eqref{mat:Lamb}
and $\bI$ is
the $M\times M$ identity matrix, and
$\boldsymbol{1}=(1,\cdots,1)^T\in\bbR^M$.
Let $\bW = \diag\{\omega_i;~ i=1, \dots, M\}$, since $\bW$ is
independent of $\bar{y}$ and $\bW$ and $\bLambda$ are both diagonal
matrices, the upper formulation can be rewritten as
\begin{equation}\label{eq:discrete-matrix}
 K_0 \bLambda \bW \boldsymbol{1} + \bLambda
 \od{\bW\bZ(\bar{y})}{\bar{y}} = \frac{1}{\Kn} \left( 
  \bW\boldsymbol{1}\bomega^T\bW^{-1} - \bI \right) \bW\bZ(\bar{y}).
\end{equation}

Noticing $\bar{\xi}_i$, $i=1,\cdots,M$ are Gauss-Hermite integral
points, we have $\He_M(\bar{\xi}_i)=0$, which indicates
$\bLambda = \bR^{-1} \bM \bR$, where $\bM$ and $\bR$ are defined in
\eqref{eq:def_MQ} and \eqref{mat:eigen-vec}, respectively.
The originality of the Hermite polynomial indicates
$\sum_{i=1}^Mw_i\He_j(\bar{y}_i)=\delta_{j,0}$, thus we have
\begin{equation*}
    \bR\bW\boldsymbol{1} = \be_1, \quad \bomega^T = \be_1^T
    \bR \bW.
\end{equation*}
Now \eqref{eq:discrete-matrix} can be rewritten as
\begin{equation}\label{eq:discrete}
    \begin{aligned}
        K_0 \be_2 + \bM \od{[\bR\bW \bZ(\bar{y})]}{\bar{y}} &=
        \bM \od{V}{\bar{y}}\\
        &= \frac{1}{\Kn} \bR \left(\bW\boldsymbol{1}\bomega^T\bW^{-1}
        - \bI \right) \bW\bZ(\bar{y})\\
        &= \frac{1}{\Kn}
        \left(\bR\bW\boldsymbol{1}\bomega^T\bW^{-1}\bR^{-1}
        - \bI \right) [\bR\bW\bZ(\bar{y})]\\
        &=\frac{1}{\Kn}\left( \be_1\be_1^T-\bI \right)
        [\bR\bW\bZ(\bar{y})]\\
        &=-\frac{1}{\Kn}\bQ [\bR\bW\bZ(\bar{y})]
        = -\frac{1}{\Kn} \bQ V,
    \end{aligned}
\end{equation}
where $\bQ$ is defined in \eqref{eq:def_MQ} and
it is readily shown that $V$ can be written as
\[
V = \left\{
\begin{array}{l}
  \bar{u}_1(\bar{y}) = K_0\bar{y} + \displaystyle\sum_{j=1}^M \omega_j
  Z^j(\bar{y}), \\
  \bar{f}_{e_1+ie_2} = (\bR\bW \bZ(\bar{y}))_{i+1} = \dfrac{1}{i!}
  \displaystyle\sum_{j=1}^M
  \omega_j\He_i(\bar{\xi}^j)Z^j(\bar{y}), \quad i=1,\dots,M-1,
\end{array} \right.
\]
which is the discrete form of \eqref{eq:integralu} and
\eqref{eq:integralZ}.

Similar discretization can be carried out for boundary condition
\eqref{eq:William-bc}. It can be written as
\[
Z(0, \bar{\xi}_i) = (1 - \chi) Z(0, -\bar{\xi}_i),
\quad (i=1, \dots, \lfloor \frac{M}{2} \rfloor). 
\]
Since the zeros of Hermite polynomials are symmetric, the equation
\[
\omega_i Z(0, \bar{\xi}_i) = (1- \chi) \omega_j Z(0, \bar{\xi}_j),
\quad (i = 1, \dots, \lfloor \frac{M}{2} \rfloor)
\]
has to be satisfied for $j = M+1-i$. Thus we have
\begin{equation}\label{bc:discrete}
\bH_{\chi} \bW \bZ(0) = 0,
\end{equation}
where $\bZ(0) = (Z(0, \bar{\xi}_1), \dots, Z(0, \bar{\xi}_M))^T$ and
\[
\begin{aligned}
&\text{when $M$ is even:}~
\bH_{\chi} = \left(
\begin{array}{cccccc}
1 & & & & & \chi-1 \\
& \ddots & &  & \iddots & \\
 & & 1 & \chi-1 & &
\end{array} \right )_{\frac{M}{2} \times M} ,\\
&\text{when $M$ is odd:}~
\bH_{\chi} = \left(
\begin{array}{ccccccc}
1 & & & 0 & & & \chi-1 \\
& \ddots & & \vdots &  & \iddots & \\
 & & 1 & 0 &\chi-1 & &
\end{array} \right )_{\lfloor \frac{M}{2} \rfloor \times M}.
\end{aligned}
\]
Let
\[
    \bK_v = \frac{1}{\chi} \left(
    \begin{array}{cccc}
        \bar{\xi}_1 & \bar{\xi}_2 & \dots & \bar{\xi}_{\lfloor \frac{M}{2} \rfloor}\\
        \bar{\xi}^3_1 & \bar{\xi}^3_2 & \dots & \bar{\xi}^3_{\lfloor \frac{M}{2} \rfloor}\\
        \vdots & \vdots & \ddots & \vdots\\
        \bar{\xi}^{2\lfloor \frac{M}{2} \rfloor-1}_1 & \bar{\xi}^{2\lfloor
            \frac{M}{2} \rfloor-1}_2 & \dots & \bar{\xi}^{2\lfloor \frac{M}{2}
        \rfloor-1}_{\lfloor \frac{M}{2} \rfloor} 
    \end{array} \right)_{\lfloor \frac{M}{2} \rfloor \times \lfloor
    \frac{M}{2} \rfloor}.
\]
Since $\bar{\xi}_1,\cdots,\bar{\xi}_{\lfloor\frac{M}{2}\rfloor}$ are
distinct, $\bK_v$ is invertible due to the invertibility of
Vandermonde matrix.
Let $\tilde{\bR} = (\tilde{r}_{ij})_{M\times M}$ with $\tilde{r}_{ij}
= \He_{i-1}(\lambda_j)$, $i, j = 1,\dots,M,$ then from
\eqref{mat:eigen-vec} we have 
$\bR = \diag \{1,1,\frac{1}{2!},\dots,\frac{1}{(M-1)!}\} \cdot
\tilde{\bR}$. 
Using the orthogonality of Hermite polynomials
\[
    \frac{1}{\sqrt{2\pi}}\int_{\bbR} \He_j(x) \He_k(x)
    \exp\left(-\frac{x^2}{2}\right) \dd x = j!\delta_{jk},
\]
we have 
\[
    \bW \tilde{\bR}^T\bR = \bI.
\]
Then we multiply matrix \eqref{bc:discrete} by $\bK_v$, and the matrix
form of boundary condition becomes
\begin{equation}\label{bc:mat-form}
    \bK_v \bH_{\chi} \bW \bZ(0) = [\bK_v \bH_{\chi} \bW \tilde{\bR}^T]
    \cdot [\bR\bW \bZ(0)] = 0. 
\end{equation}
The discretization of $S(l,m)$ in \eqref{eq:integralS} is
\[
\begin{aligned}
    S(l,m) &= \frac{1}{\chi} \sum_{j=1}^{\lfloor \frac{M}{2} \rfloor}
    \{\bar{\xi}_j^l \omega_j 
    [\He_m(\bar{\xi}_j) - (1 - \chi)\He_m(-\bar{\xi}_j)]\} \\
    &=\frac{1}{\chi} (\bar{\xi}_1^l, \bar{\xi}_2^l, \dots,
    \bar{\xi}_{\lfloor \frac{M}{2} \rfloor}^l) 
    \cdot \bH_{\chi} \bW \cdot (\He_m(\bar{\xi}_1),
    \He_m(\bar{\xi}_2),\dots, \He_m(\bar{\xi}_M))^T,
\end{aligned}
\]
then
\[
    \bK_v \bH_{\chi} \bW \tilde{\bR}^T = \left(
    \begin{array}{cccc}
        S(1,0)                              & S(1,1) & \cdots & S(1,M-1) \\
        S(3,0)                              & S(3,1) & \cdots & S(3,M-1) \\
        \vdots                              & \vdots & \ddots & \vdots   \\
        S(2\lfloor \frac{M}{2} \rfloor-1,0) & S(2\lfloor \frac{M}{2} \rfloor-1,1)
                                            & \cdots & S(2\lfloor \frac{M}{2} \rfloor-1,M-1)
    \end{array} \right).
\]
And \eqref{bc:mat-form} is then turned into
\[
\left(
\begin{array}{cccc}
S(1,0) & S(1,1) & \cdots & S(1,M-1)\\
S(3,0) & S(3,1) & \cdots & S(3,M-1)\\
\vdots & \vdots & \ddots & \vdots\\
S(2\lfloor \frac{M}{2} \rfloor-1,0) & S(2\lfloor \frac{M}{2}
\rfloor-1,1) & \cdots & S(2\lfloor \frac{M}{2} \rfloor-1,M-1) 
\end{array} \right) \cdot V^{(0)} = 0,
\]
which is same as the boundary condition in \eqref{bc:linear}.



\section{Quantity Validification}
In this section, we numerically study the convergence of the solutions
of the linearized HME to that of the linearized Boltzmann equation,
and Knudsen layer effect of the velocity and effective viscosity, and
compare them with the existing results. In all the tests, high
precision computation in Maple\footnote{Maple is a trademark of
Waterloo Maple Inc.} is used to reduce the numerical error.

\subsection{Convergence in moment order}\label{sec:convergenceNum}
In order to compare the results with linearized Boltzmann equation
\cite{Williams2001}, we normalized the velocity in
\eqref{eq:sol_u1_formal} as
\begin{equation}\label{eq:refer_velo}
    \tilde{u}(\bar{y}) = -\Kn\dfrac{\bar{u}}{\bar{\sigma}_{12}}
    = \bar{y} + \frac{\Kn}{\bar{\sigma}_{12}} \left(2\sum_{i = 1}^{\lfloor
        \frac{M-2}{2} \rfloor} \hat{c_i}
        \exp\left(-\frac{\bar{y}}{\Kn\hat{\lambda}_i}\right) - c_0\right).
\end{equation}
The normalized velocity can be split into three parts
\cite{Williams2001,Siewert2001}
\begin{equation}\label{eq:defe_velo}
    \tilde{u}(\bar{y})
    = \bar{y} + \zeta - \tilde{u}_d(\bar{y}),
\end{equation}
where $\tilde{u}_d(\bar{y})$ is the velocity defect, satisfying
$\lim\limits_{\bar{y} \to +\infty} \tilde{u}_d(\bar{y}) = 0$, and
$\zeta$ is the slip coefficient, which is
\begin{equation}\label{eq:slip_coefficient_approx}
    \zeta =
    \lim_{\bar{y}\to+\infty}(\tilde{u}(\bar{y})+\tilde{u}_d(\bar{y})-\bar{y})
    = -\Kn \cdot \frac{c_0}{\bar{\sigma}_{12}}.
\end{equation}
Then the velocity defect is
\begin{equation}
    \tilde{u}_d(\bar{y}) = 
    2\Kn \sum_{i = 1}^{\lfloor
        \frac{M-2}{2} \rfloor} \frac{\hat{c_i}}{\bar{\sigma}_{12}}
        \exp\left(-\frac{\bar{y}}{\Kn\hat{\lambda}_i}\right).
\end{equation}
Here we notice that there is always a factor $\bar{\sigma}_{12}$ in
the expression of $\hat{c}_i$ and $c_0$ in \eqref{eq:sol_u1_formal}.
In this subsection, we fix $\Kn = 1/\sqrt{2}$ as a constant
for convenience. Next we study the convergence of the velocity defect
and slip coefficient, respectively.

\begin{figure}[ht]
  \centering 
  \subfigure[]{
    \begin{overpic}[width=0.45\textwidth,clip]{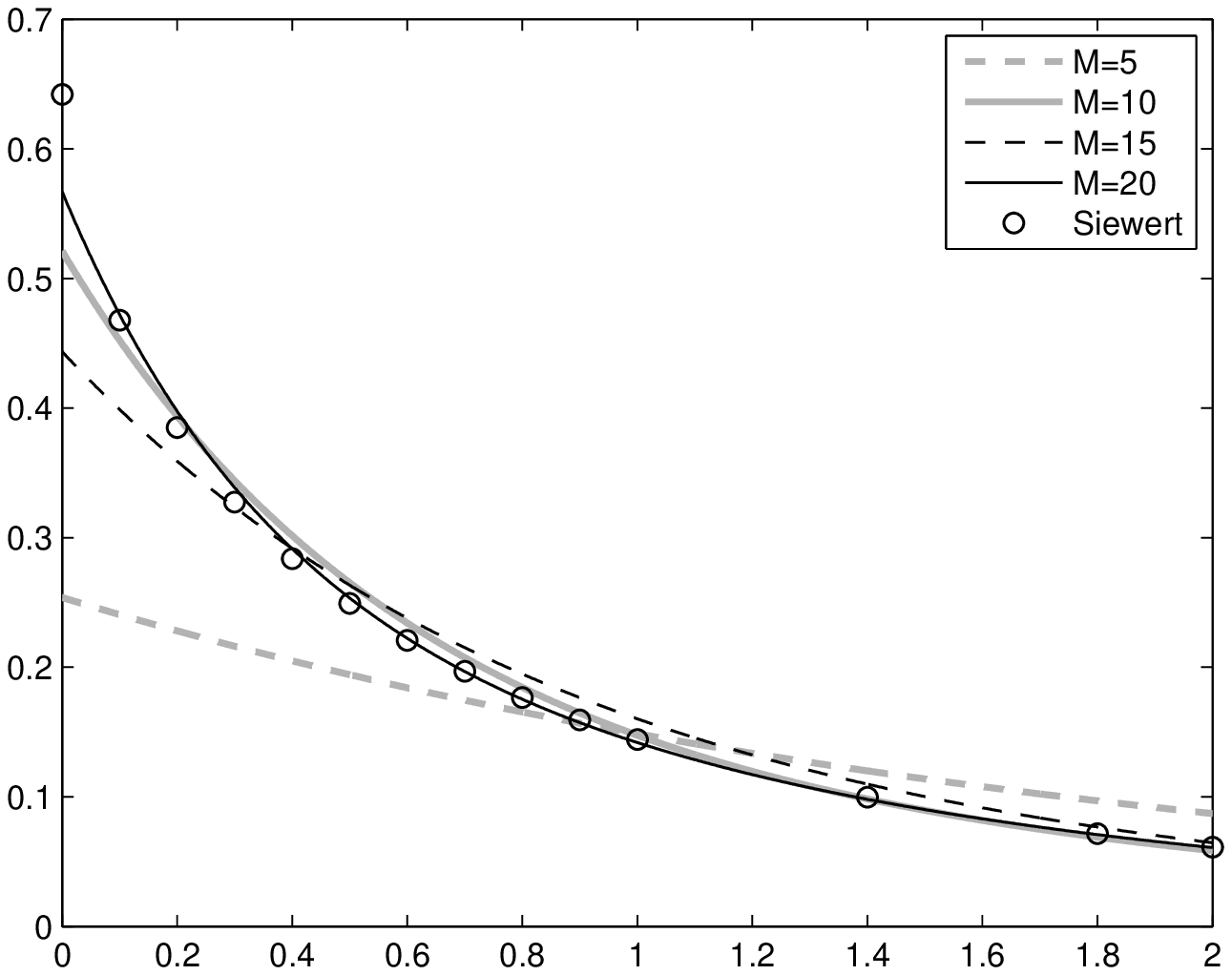}
        \put(-8,39){\small$\tilde{u}_d(\bar{y})$}
        \put(54,-1){\small$\bar{y}$}
    \end{overpic}}\qquad
  \subfigure[]{
    \begin{overpic}[width=0.45\textwidth,clip]{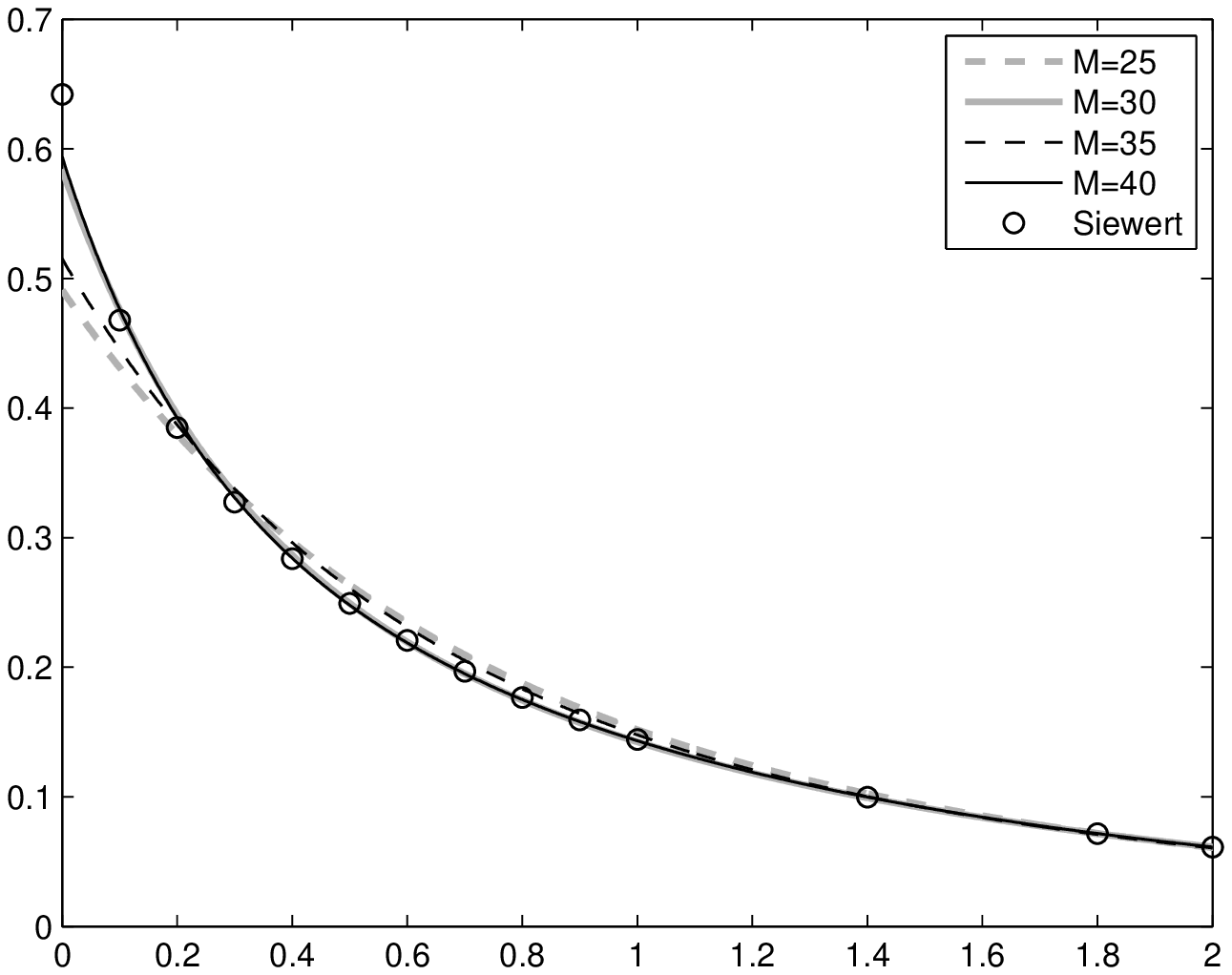}
        \put(-8,39){\small$\tilde{u}_d(\bar{y})$}
        \put(54,-1){\small$\bar{y}$}
    \end{overpic}}
  \subfigure[]{
    \begin{overpic}[width=0.45\textwidth,clip]{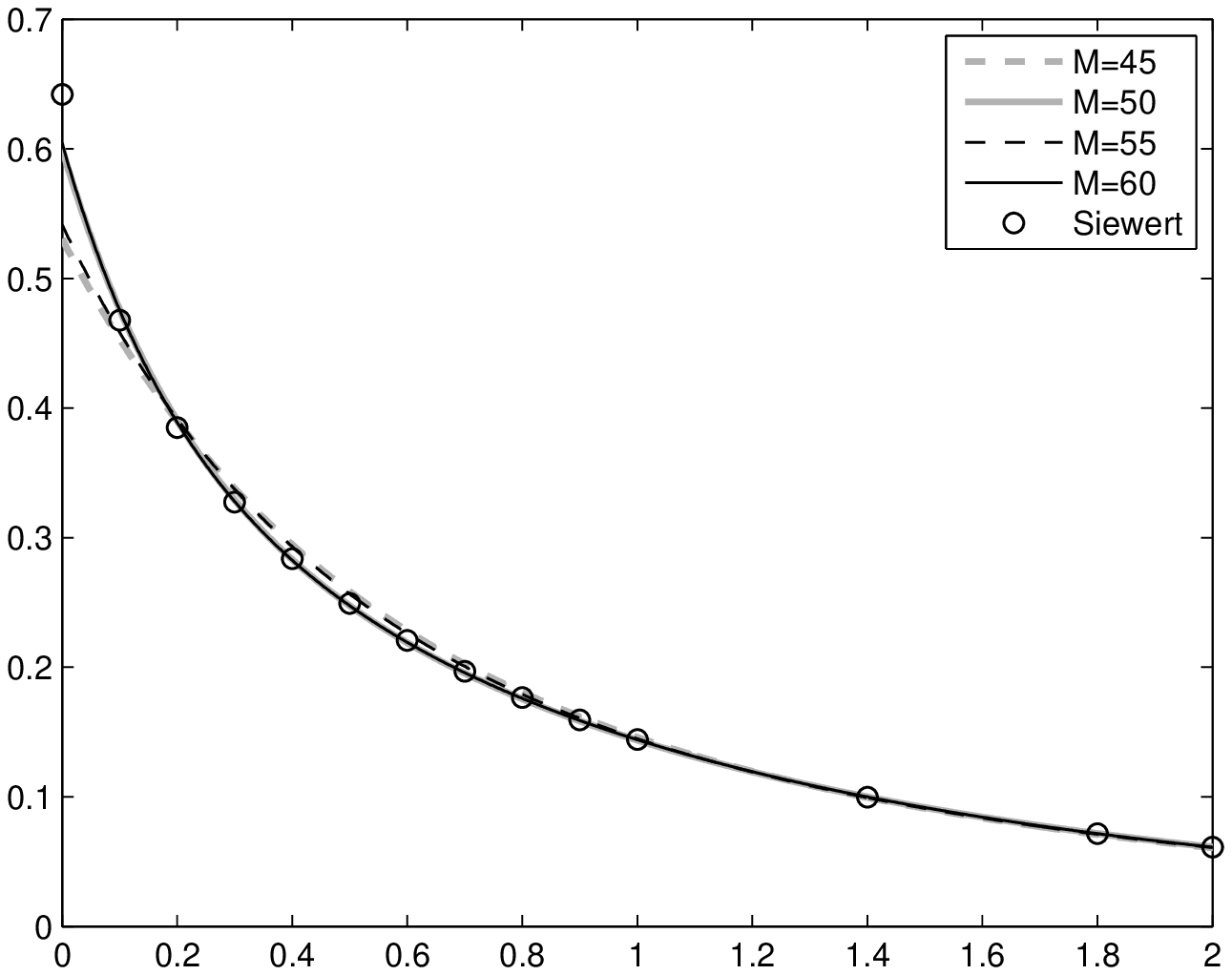}
        \put(-8,39){\small$\tilde{u}_d(\bar{y})$}
        \put(54,-1){\small$\bar{y}$}
    \end{overpic}}\qquad
  \subfigure[]{
    \begin{overpic}[width=0.45\textwidth,clip]{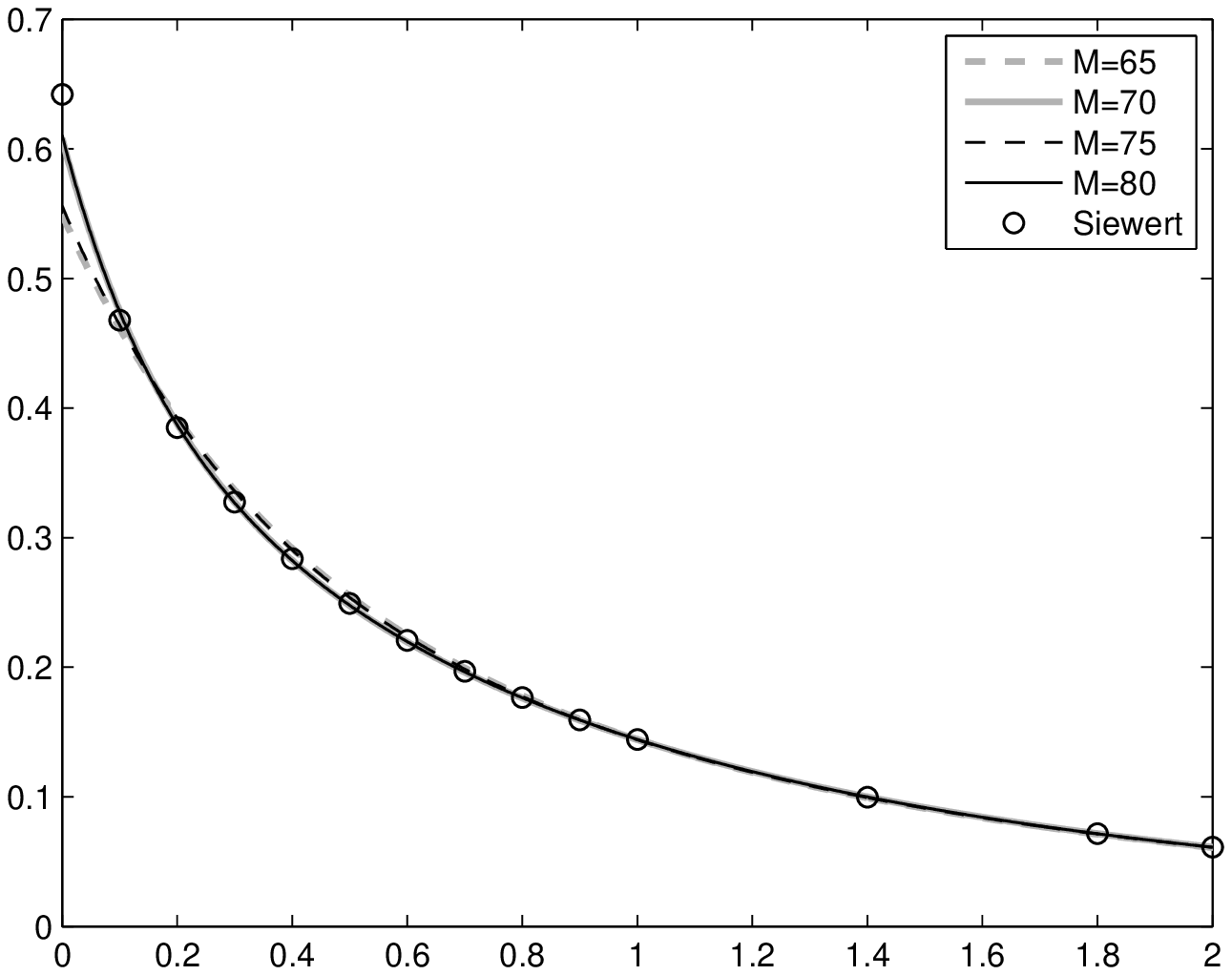}
        \put(-8,39){\small$\tilde{u}_d(\bar{y})$}
        \put(54,-1){\small$\bar{y}$}
    \end{overpic}}
  \caption{Profile of the defect velocity $\tilde{u}_d(\bar{y})$ of
  the linearized HME for different $M$ with $\chi=0.1$. The reference
  solution is Siewert's result in \cite{Siewert2001} for linearized
  BGK model.}
  \label{fig:diff_M_1}
\end{figure}

\begin{figure}[ht]
  \centering 
  \subfigure[]{
    \begin{overpic}[width=0.45\textwidth,clip]{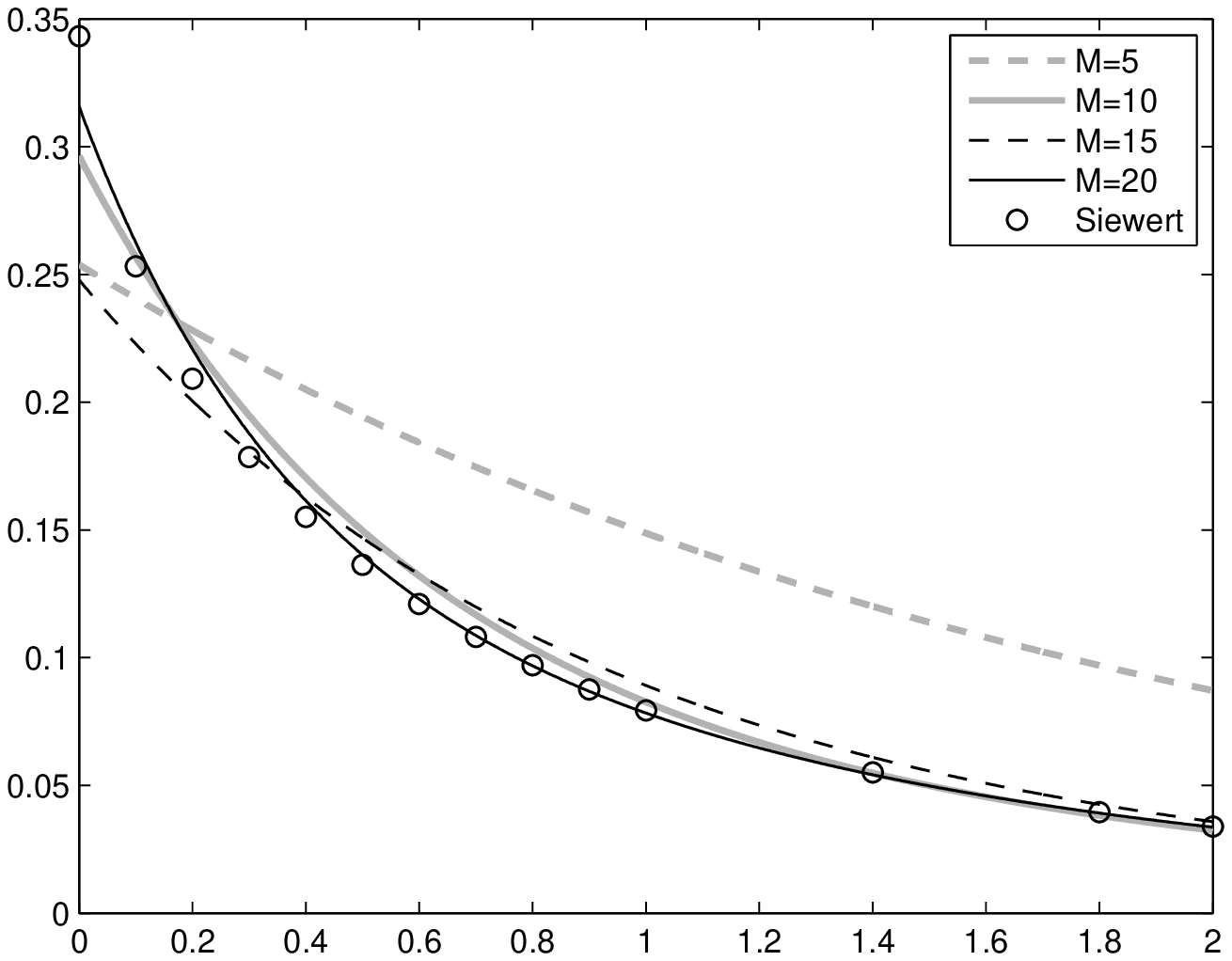}
        \put(-8,39){\small$\tilde{u}_d(\bar{y})$}
        \put(54,-1){\small$\bar{y}$}
    \end{overpic}}\qquad
  \subfigure[]{
\begin{overpic}[width=0.45\textwidth,clip]{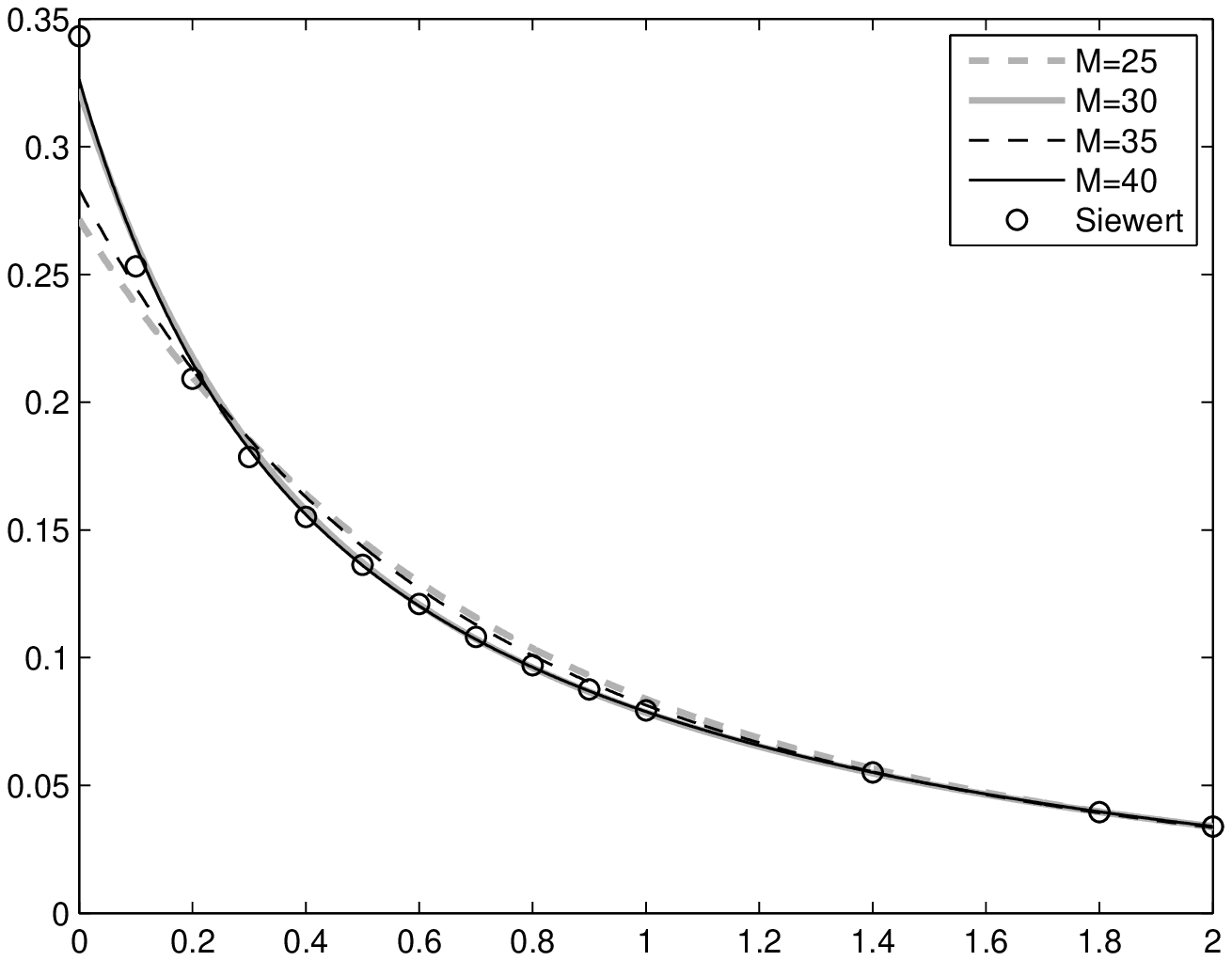}
        \put(-8,39){\small$\tilde{u}_d(\bar{y})$}
        \put(54,-1){\small$\bar{y}$}
    \end{overpic}}
  \subfigure[]{
\begin{overpic}[width=0.45\textwidth,clip]{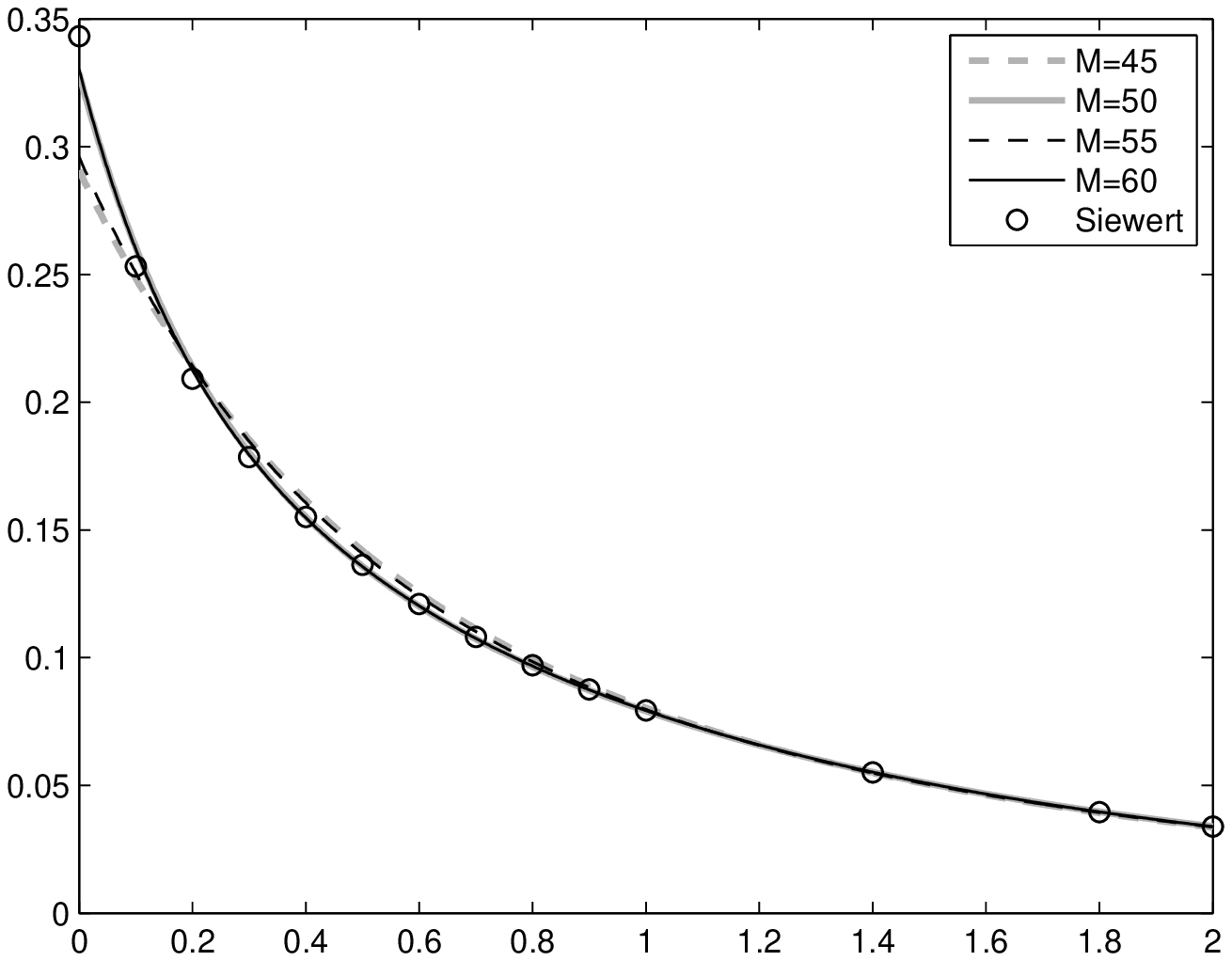}
        \put(-8,39){\small$\tilde{u}_d(\bar{y})$}
        \put(54,-1){\small$\bar{y}$}
    \end{overpic}}\qquad
  \subfigure[]{
\begin{overpic}[width=0.45\textwidth,clip]{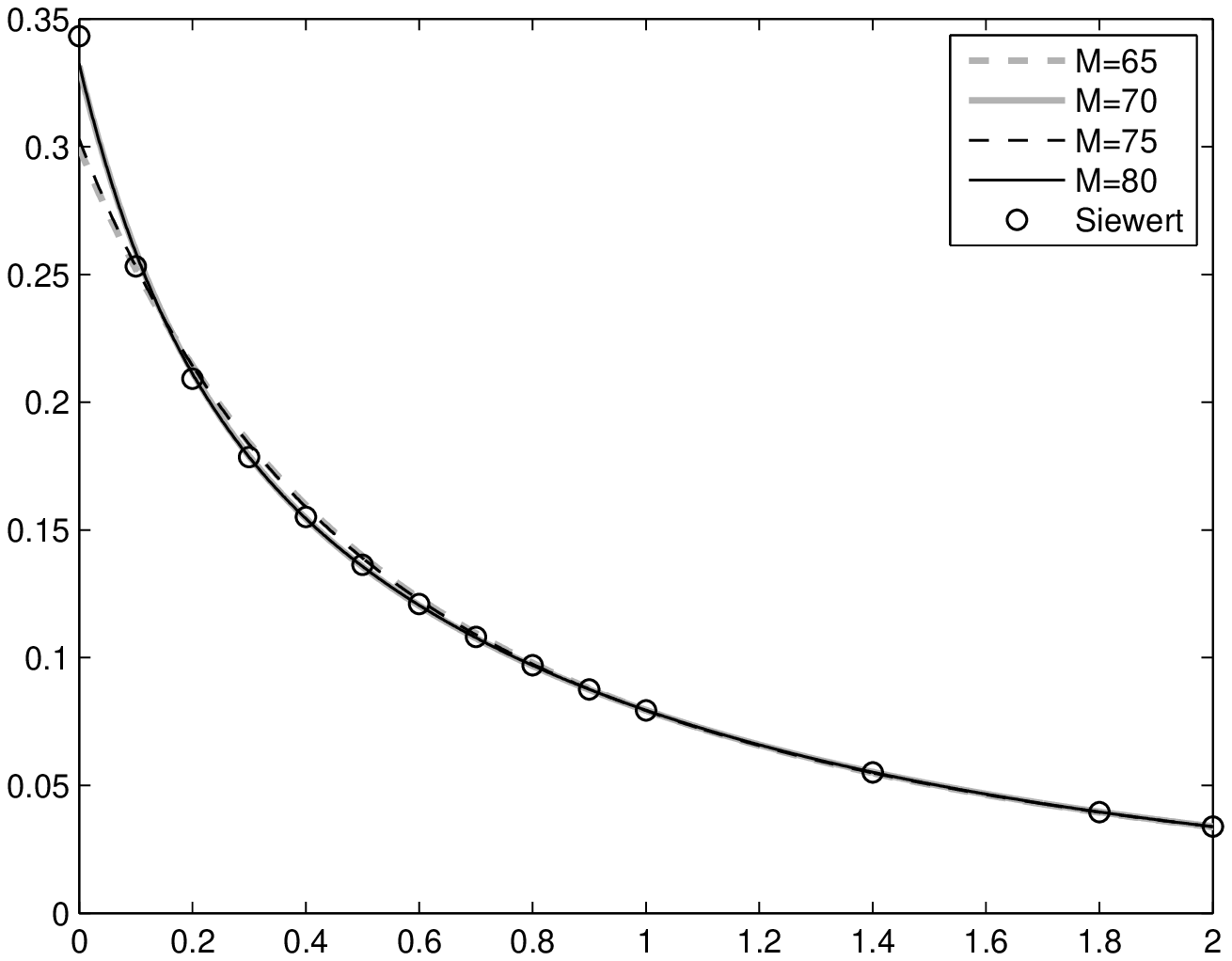}
        \put(-8,39){\small$\tilde{u}_d(\bar{y})$}
        \put(54,-1){\small$\bar{y}$}
    \end{overpic}}
  \caption{Profile of the defect velocity $\tilde{u}_d(\bar{y})$ of
  the linearized HME for different $M$ with $\chi=0.9$. The reference
  solution is Siewert's result in \cite{Siewert2001} for linearized
  BGK model.}
  \label{fig:diff_M_9}
\end{figure}

\begin{figure}[ht]
    \centering
    \begin{overpic}[width=0.9\textwidth,clip]{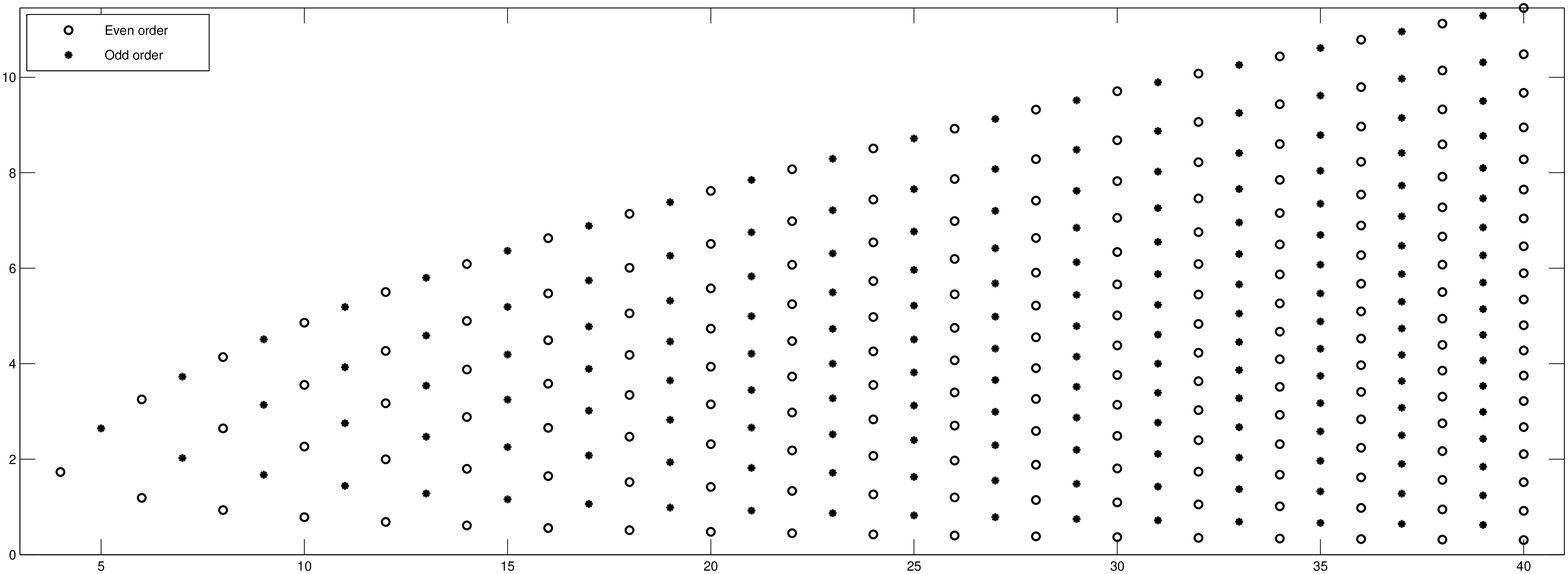}
        \put(-2,22){\small$\hat{\lambda}_i$}
        \put(52,-2){\small$M$}
    \end{overpic}
    \caption{\label{fig:hatlambdai}The value $\hat{\lambda}_i$ for
different $M$.}
\end{figure}
For the velocity defect $\tilde{u}_d(\bar{y})$, the analytical results
with $M$ ranging from $5$ to $80$ are presented in Fig.
\ref{fig:diff_M_1} for $\chi = 0.1$ and Fig. \ref{fig:diff_M_9} for
$\chi = 0.9$, which are compared with the Siewert's numerical results
in \cite{Siewert2001} for the linearized BGK model.  It is clear that
the results of the linearized HME converge to Siewert's result as $M$
increasing, which is consistent with the theoretical analysis in Sec.
\ref{sec:convergence}.
Meanwhile, one can find that the defect velocity of even order
converges faster to the reference solution than that of odd order.
This can be understood based on the smallest width of the boundary
layer, which is represented by $w_M:=\min\{\hat{\lambda}_i:
i=1,\cdots,\lfloor \frac{M}{2}-1\rfloor\}$. The smaller $w_M$, the
closer of the defect velocity of the linearized HME to the reference
solution.  Fig.  \ref{fig:hatlambdai} gives all the
$\hat{\lambda}_i$ for $M$ ranging from $3$ to $40$. One can observe
that $w_M$ for even $M$ is quite smaller than that for the adjacent
odd $M$.

Moreover, comparing with Fig. \ref{fig:diff_M_1} and
\ref{fig:diff_M_9}, one can find that for a given $M$, the relative
error in Fig. \ref{fig:diff_M_1} is a little larger than that of Fig.
\ref{fig:diff_M_9}. Actually, for smaller $\chi$, the diffusion
interaction between gas and the wall turns weak, then the distribution
function is expected to be more far from the equilibrium, which
indicates more moment is needed.

For the slip coefficient $\zeta$, the analytical results for different
$M$ are plotted in Fig. \ref{fig:zeta}. Similar convergence can be
readily observed in Fig. \ref{fig:zeta}.
All the phenomena observed in Fig. \ref{fig:diff_M_1} and
\ref{fig:diff_M_9} are also valid in Fig. \ref{fig:zeta}.
\begin{figure}[ht]
  \centering 
  \subfigure[$\chi = 0.1$]{
    \begin{overpic}[width=0.3\textwidth,clip]{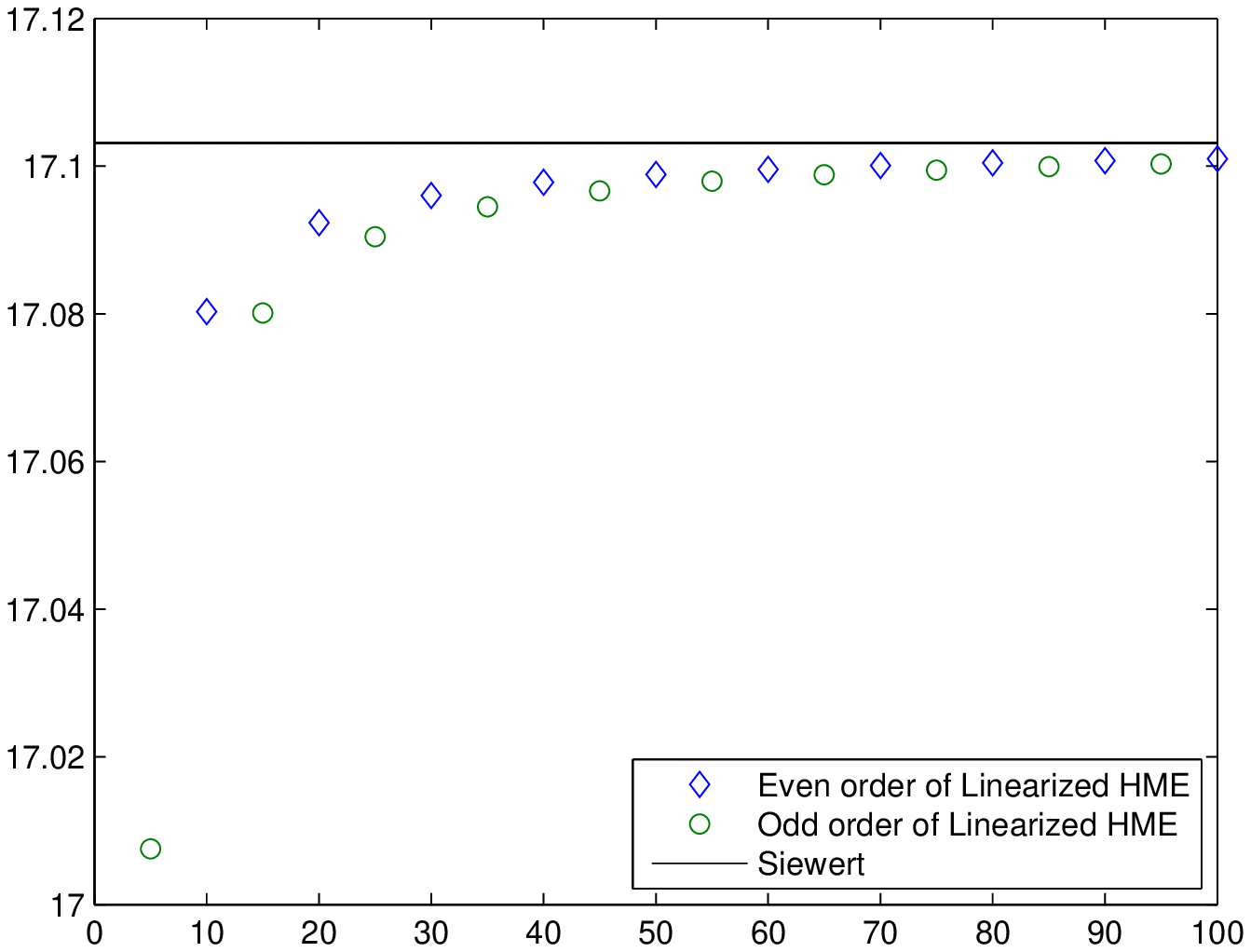}
      \put(0,41){\small$\zeta$}
      \put(50,0){\tiny$M$}
      \end{overpic}}\quad
  \subfigure[$\chi = 0.3$]{
    \begin{overpic}[width=0.3\textwidth,clip]{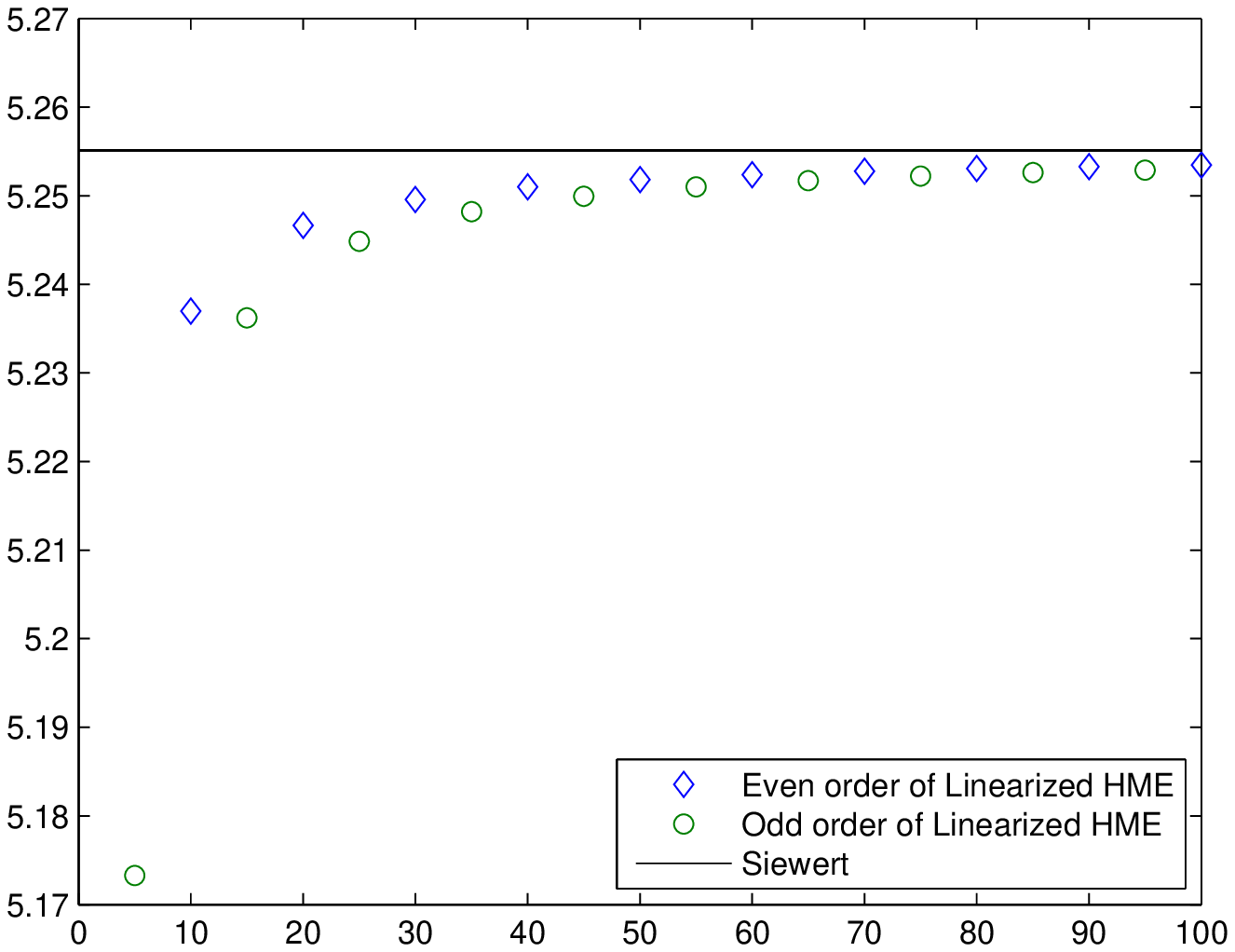}
      \put(0,41){\small$\zeta$}
      \put(50,0){\tiny$M$}
      \end{overpic}}\quad
  \subfigure[$\chi = 0.5$]{
    \begin{overpic}[width=0.3\textwidth,clip]{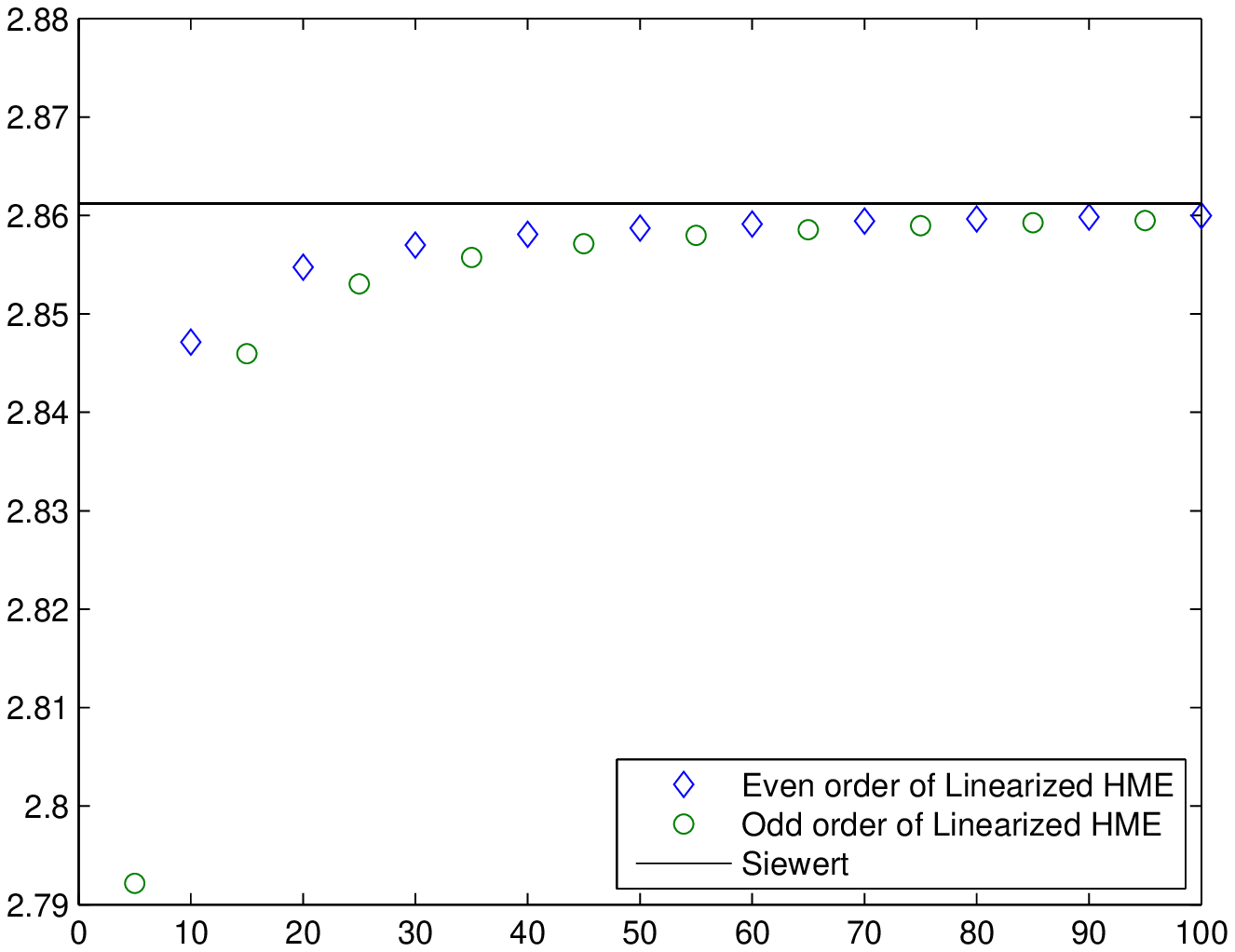}
      \put(0,41){\small$\zeta$}
      \put(50,0){\tiny$M$}
      \end{overpic}}
  \subfigure[$\chi = 0.7$]{
    \begin{overpic}[width=0.3\textwidth,clip]{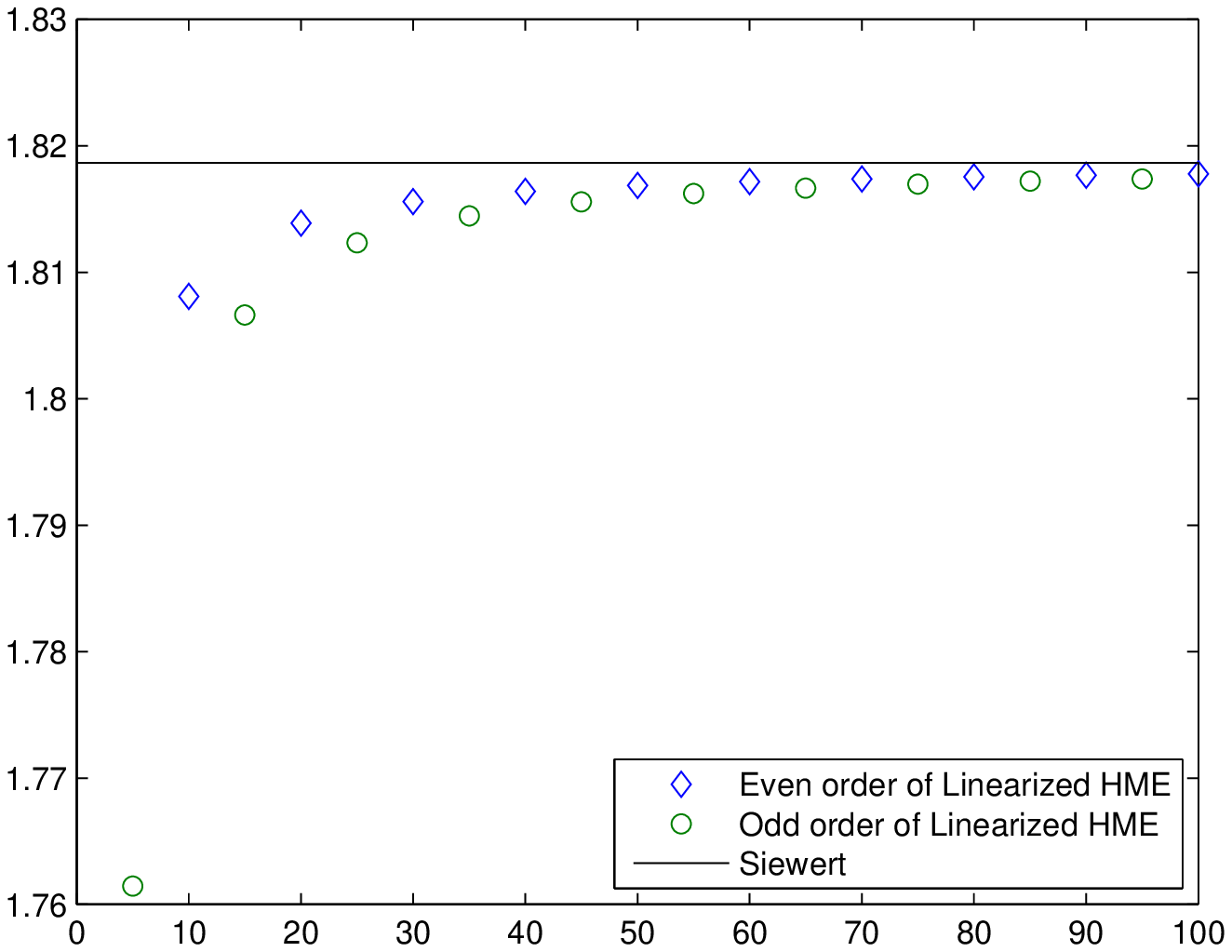}
      \put(0,41){\small$\zeta$}
      \put(50,0){\tiny$M$}
      \end{overpic}}\quad
  \subfigure[$\chi = 0.9$]{
    \begin{overpic}[width=0.3\textwidth,clip]{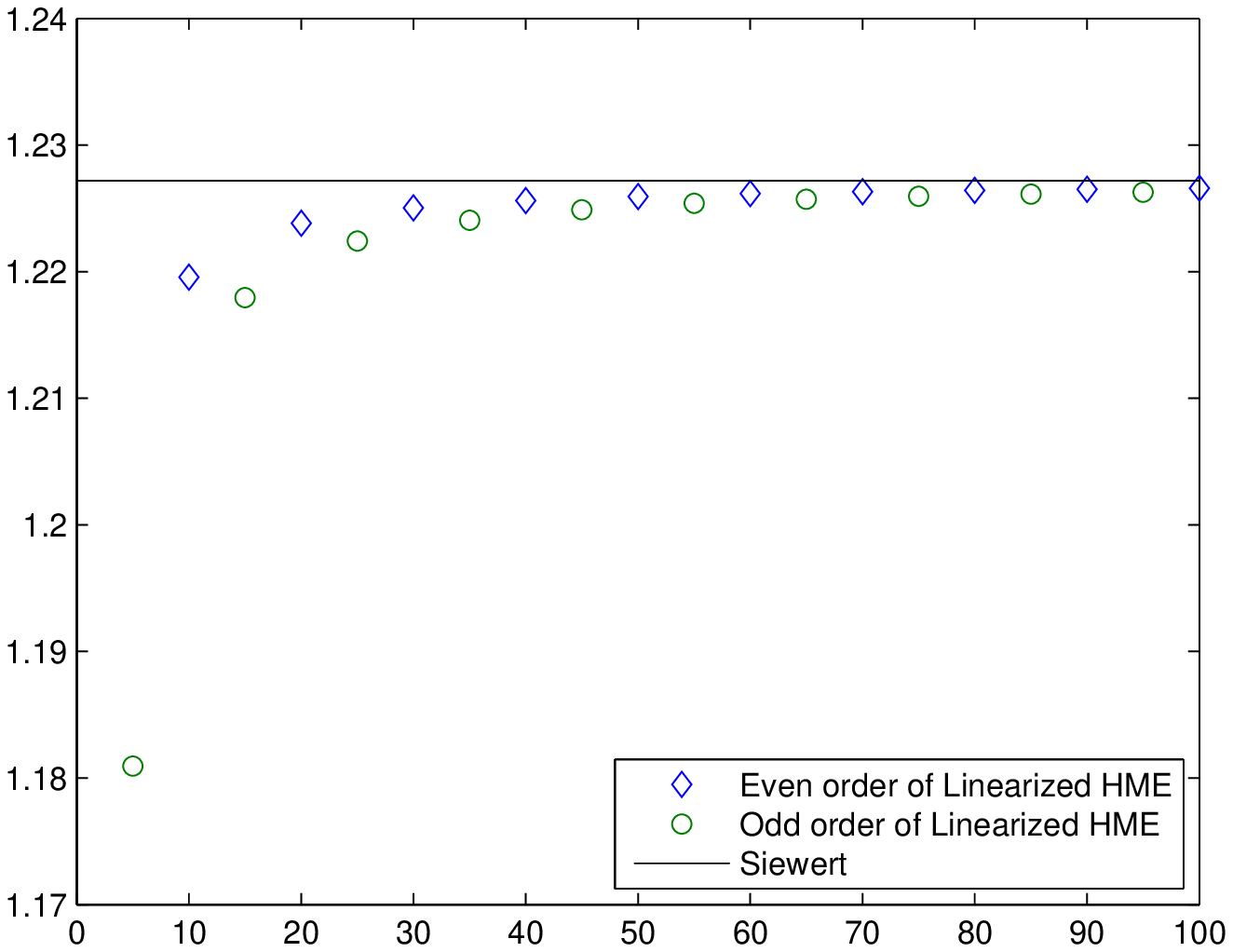}
      \put(0,41){\small$\zeta$}
      \put(50,0){\tiny$M$}
      \end{overpic}}\quad
  \subfigure[$\chi=1.0$]{
    \begin{overpic}[width=0.3\textwidth,clip]{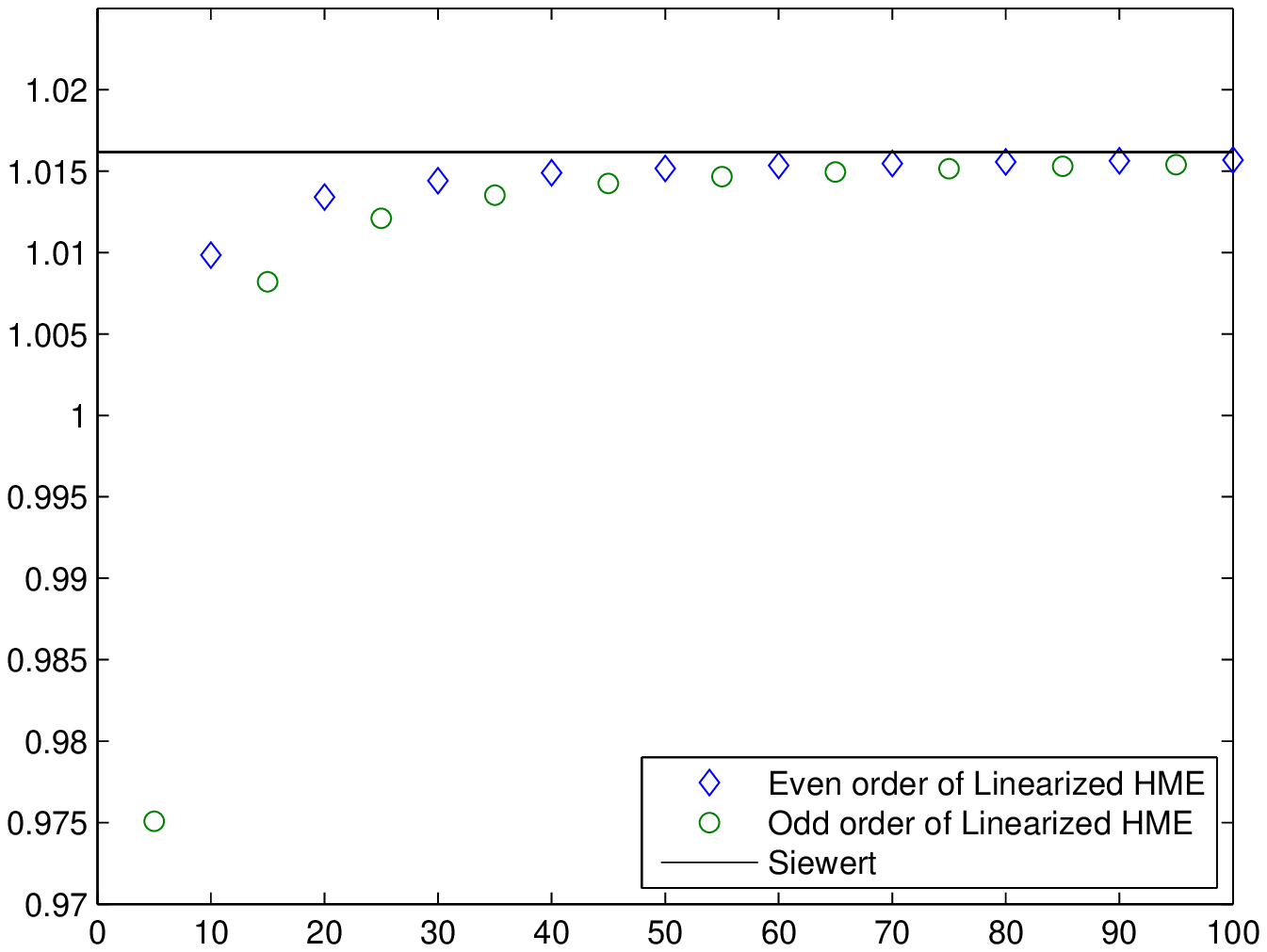}
      \put(0,41){\small$\zeta$}
      \put(50,0){\tiny$M$}
      \end{overpic}}
  \caption{Values of the slip coefficient $\zeta$ for different $M$
      and $\chi$. The reference solution is Siewert's result in
      \cite{Siewert2001} for linearized BGK model.}
      \label{fig:zeta}
\end{figure}

\subsection{Knudsen layer}
In this subsection, we study the Knudsen layer of Kramers' problem in
three aspects. The first one is the profile of the normalized velocity
$\tilde{u}(\bar{y})$ \eqref{eq:refer_velo}. For convergence, here we
also fix $\Kn$ as a constant $1 /\sqrt{2}$.  Fig.
\ref{fig:diff_chi} gives the profile of $\tilde{u}(\bar{y})$ in
\eqref{eq:refer_velo} of linearized HME with $M=8$ and $M=9$.
Compared with numerical results of linearized Boltzmann equation in
\cite{Loyalka1975}, the good agreement of the solutions of the
linearized HME in Fig. \ref{fig:diff_chi} indicates the moment system
with a small $M$ is good enough to describe the velocity profile in
the Knudsen layer.
Moreover, the value of $\tilde{u}(\bar{y})$ increases, as $\chi$
decreasing. This is because the coefficients $\hat{c}_i$ and $c_0$ are
dependent on $\frac{2-\chi}{\chi}$.
As discussed in the Section \ref{sec:convergenceNum}, the diffusion
interaction between gas and the wall is weaker for smaller $\chi$.
\begin{figure}[ht]
    \centering
    \subfigure[$M=8$]{%
        \begin{overpic}[width=0.45\textwidth,clip]{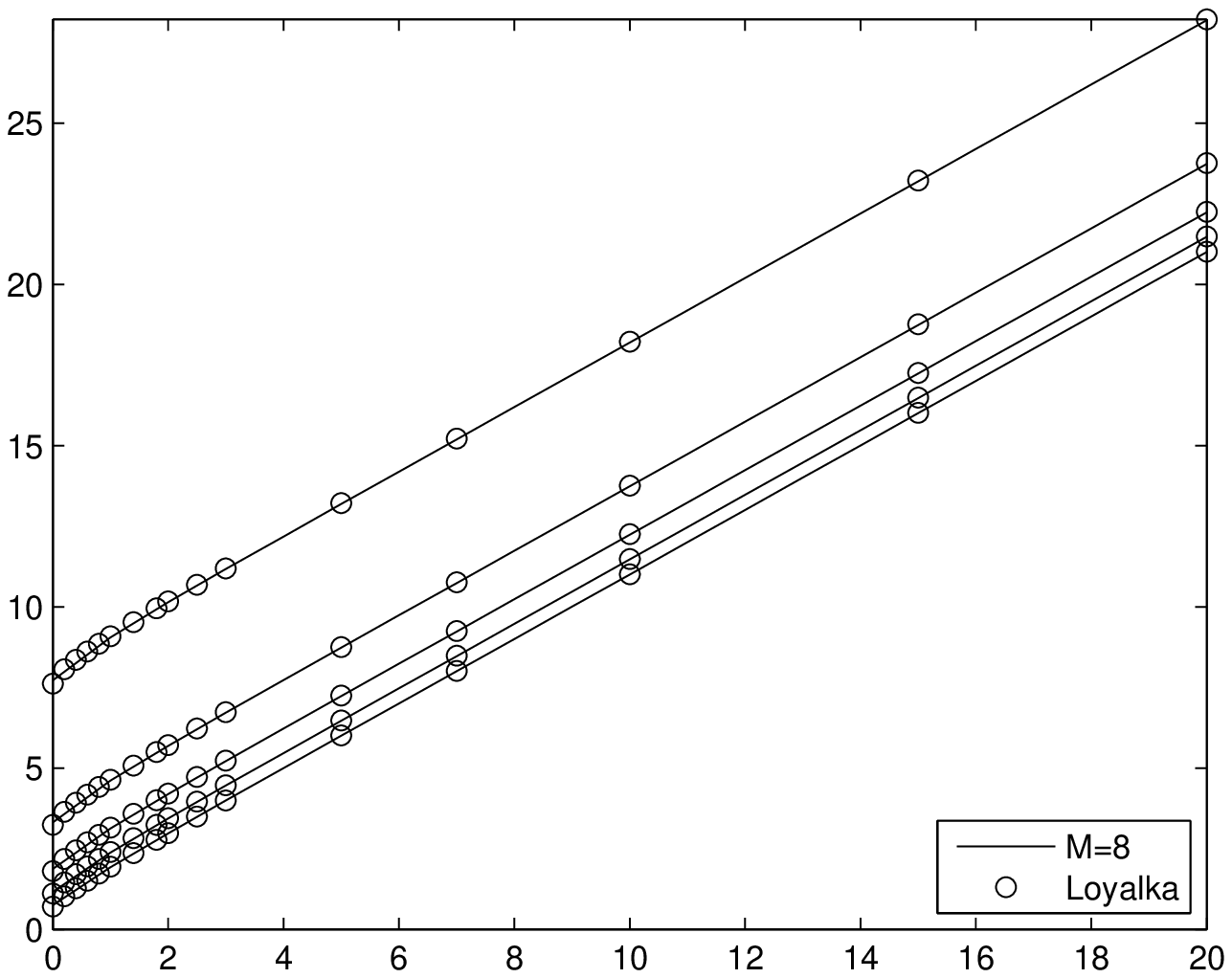}
            \put(-6,47){\small$\tilde{u}(\bar{y})$}
            \put(82,-3){\small$\bar{y}$}
            \put(60,60){\small$\downarrow$}
            \put(56,64){\small$\chi=0.2$}
            \put(54,46){\small$\downarrow$}
            \put(50,50){\small$\chi=0.4$}
            \put(54,34){$\uparrow$}
            \put(48,30){\small$\chi=0.6$}
            \put(65,38){$\uparrow$}
            \put(60,34){\small$\chi=0.8$}
            \put(80,44){$\uparrow$}
            \put(76,40){\small$\chi=1$}
        \end{overpic}
    }\qquad
    \subfigure[$M=9$]{%
        \begin{overpic}[width=0.45\textwidth,clip]{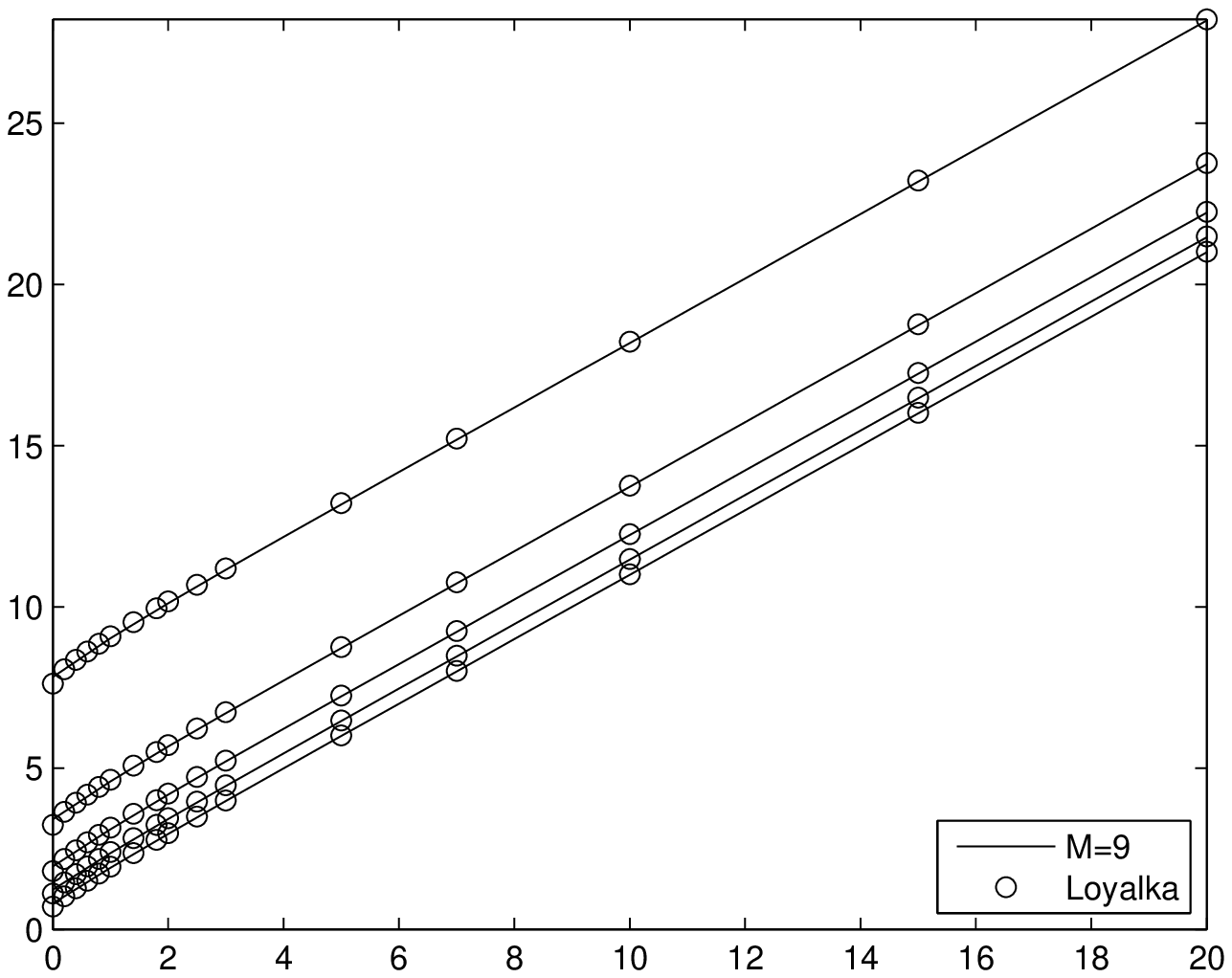}
            \put(-6,47){\small$\tilde{u}(\bar{y})$}
            \put(82,-3){\small$\bar{y}$}
            \put(60,60){\small$\downarrow$}
            \put(56,64){\small$\chi=0.2$}
            \put(54,46){\small$\downarrow$}
            \put(50,50){\small$\chi=0.4$}
            \put(54,34){$\uparrow$}
            \put(48,30){\small$\chi=0.6$}
            \put(65,38){$\uparrow$}
            \put(60,34){\small$\chi=0.8$}
            \put(80,44){$\uparrow$}
            \put(76,40){\small$\chi=1$}
        \end{overpic}
    }
    \caption{ Profile of $\tilde{u}(\bar{y})$ of the linearized HME for
    different accommodation number $\chi$. The reference solution is
    Loyalka's result in \cite{Loyalka1975}.}
    \label{fig:diff_chi}
\end{figure}

\begin{figure}[!htb]
  \centering
  \subfigure[]{
    \begin{overpic}[width=0.45\textwidth,clip]{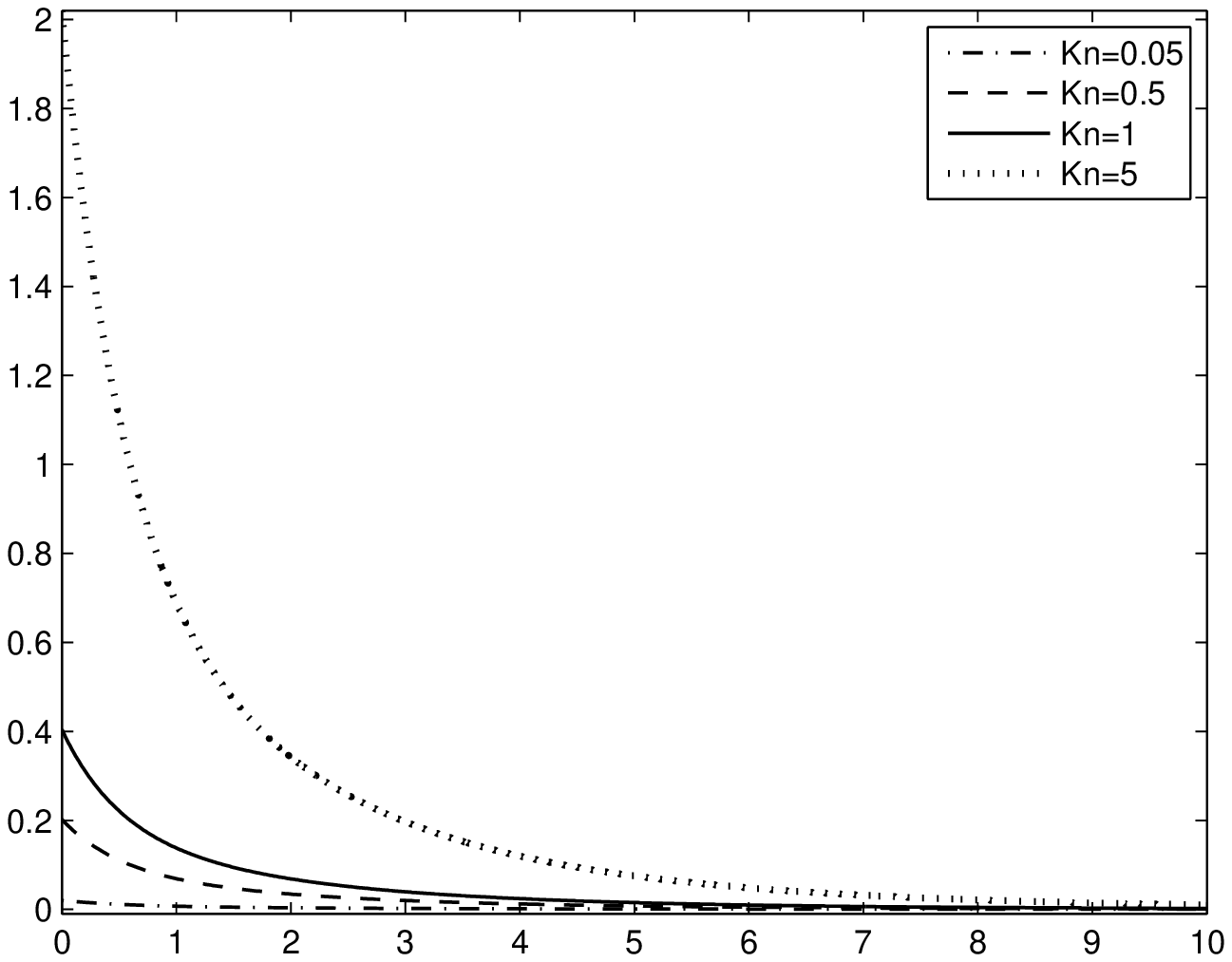}
      \put(-8,42){\small$\tilde{u}_d(\bar{y})$}
      \put(60,-3){\small$\bar{y}/\Kn$}
      \end{overpic}}\qquad
  \subfigure[]{
    \begin{overpic}[width=0.45\textwidth,clip]{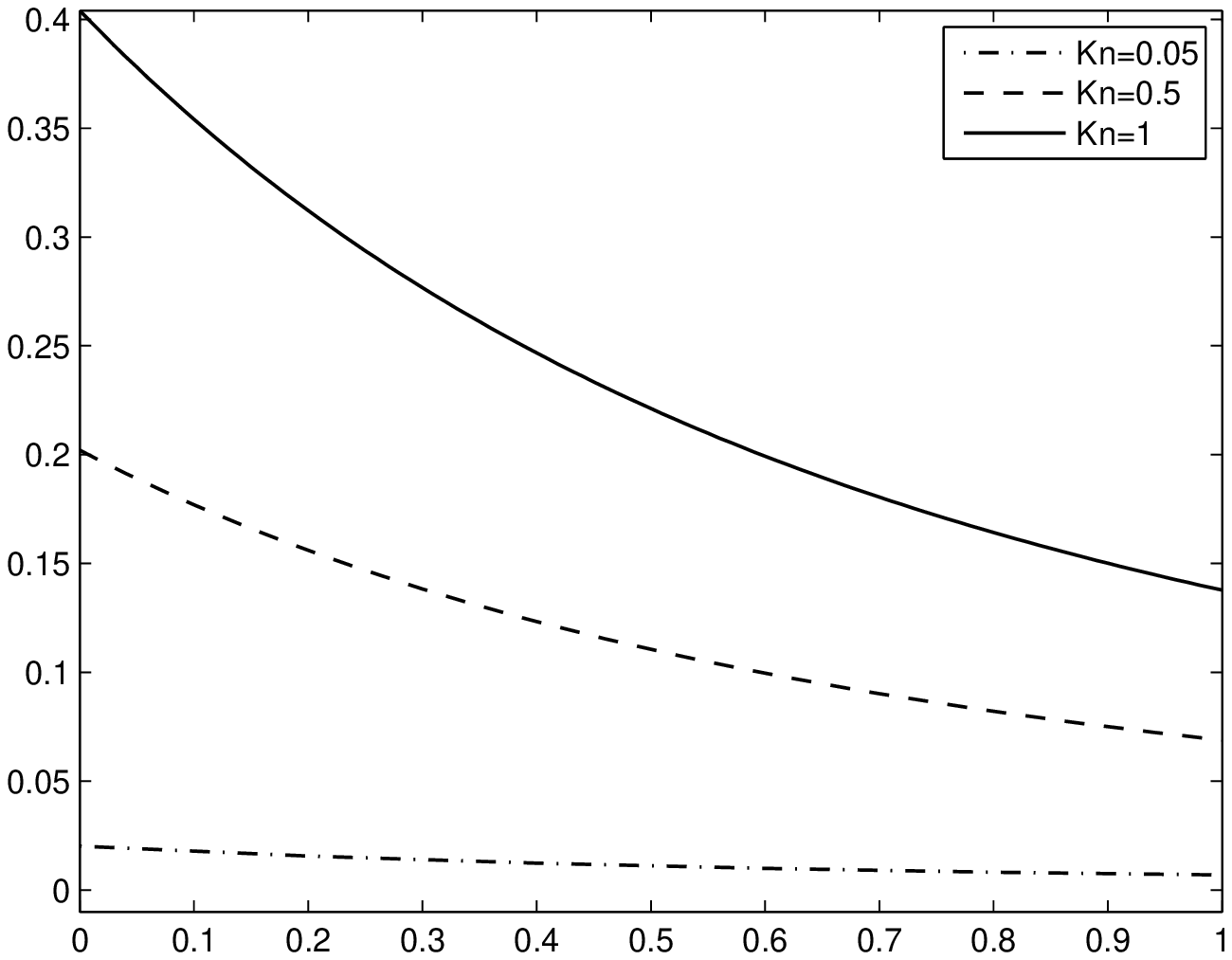}
      \put(-8,42){\small$\tilde{u}_d(\bar{y})$}
      \put(60,-3){\small$\bar{y}/\Kn$}
      \end{overpic}}
      \caption{Profile of $\tilde{u}_d(\bar{y})$ for different $\Kn$.}
  \label{fig:diff_Kn}
\end{figure}
The second one is the profile of the velocity defect
$\tilde{u}_d(\bar{y})$ in \eqref{eq:defe_velo}.
Fig. \ref{fig:diff_Kn} shows the profile of $\tilde{u}_d(\bar{y})$
with $M=20$ for different Knudsen number. 
The thickness of the Knudsen layer largens as $\Kn$ increasing and the
strength of of the Knudsen layer enhances. In practical application,
more moments are needed for large $\Kn$.

The third one is the effective viscosity. The Navier-Stokes law
indicates $\sigma_{12}=-\mu\pd{u}{y}$.  However, in the Knudsen layer,
the Navier-Stokes does not hold anymore. To describe 
the non-Newtonian behavior inherent in the Knudsen layer, we formally
write the Navier-Stokes law on the shear stress $\sigma_{12}$ as
\begin{equation}\label{eq:effective}
    \sigma_{12} = -\mu_{\mathrm{eff}}\pd{u}{y},
\end{equation}
where $\mu_{\mathrm{eff}}$ is called the ``effective viscosity''.
Since the shear stress $\sigma_{12}$ is constant in the Kramers'
problem, we have
\begin{equation}
    \frac{\mu_{\mathrm{eff}}}{\mu} 
    = -\left(\frac{\sigma_{12}}{\partial u / \partial y}\right) 
    \Big{/} \left(\frac{\lambda p_0}{\sqrt{\theta_0}}\right) 
    = - \frac{1}{\Kn} \frac{\bar{\sigma}_{12}}{\partial \bar{u} / \partial\bar{y}}
    = \frac{1}{\partial \tilde{u} / \partial\bar{y}}.
\end{equation}
Noticing the definition of the normalized velocity
\eqref{eq:refer_velo}, one can directly calculate
\begin{equation}\label{eq:eff_HME}
    \mu_{\mathrm{eff}} = \frac{\mu}{1 + \sum_{i = 1}^{\lfloor
        \frac{M-2}{2} \rfloor} c_i
        \exp\left(-\frac{\bar{y}}{\hat{\lambda}_i\Kn}\right)}, 
    \quad c_i =
    - \frac{2\hat{c}_i}{\hat{\lambda}_i \bar{\sigma}_{12}}.
\end{equation}

\begin{figure}[htb]
  \centering
  \subfigure[]{
        \begin{overpic}[width=0.45\textwidth,clip]{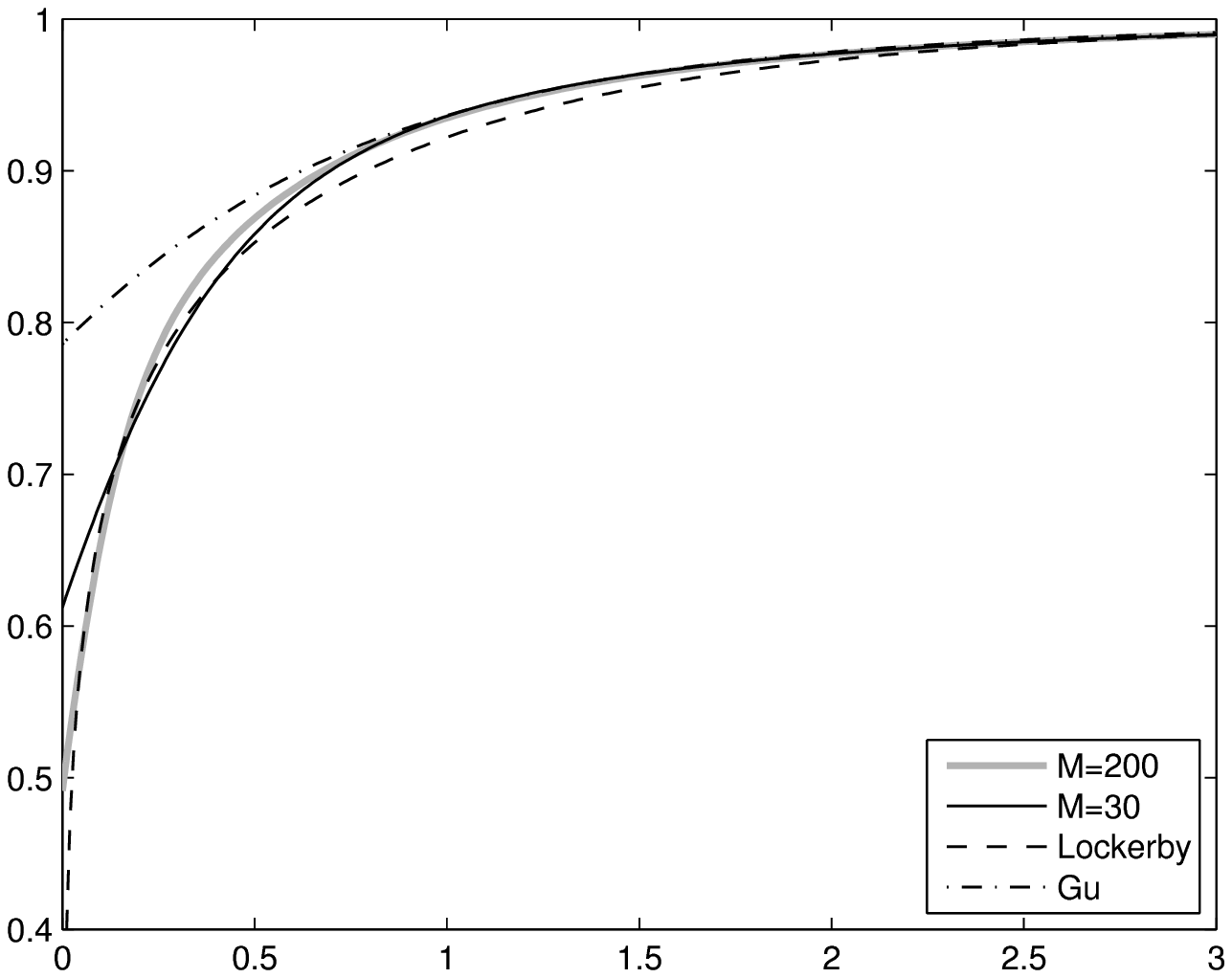}
            \put(-7,42){\small$\frac{\mu_{\mathrm{eff}}}{\mu}$}
            \put(60,-4){\small$\bar{y}$}
  \end{overpic}}
  \qquad
  \subfigure[]{
    \begin{overpic}[width=0.45\textwidth,clip]{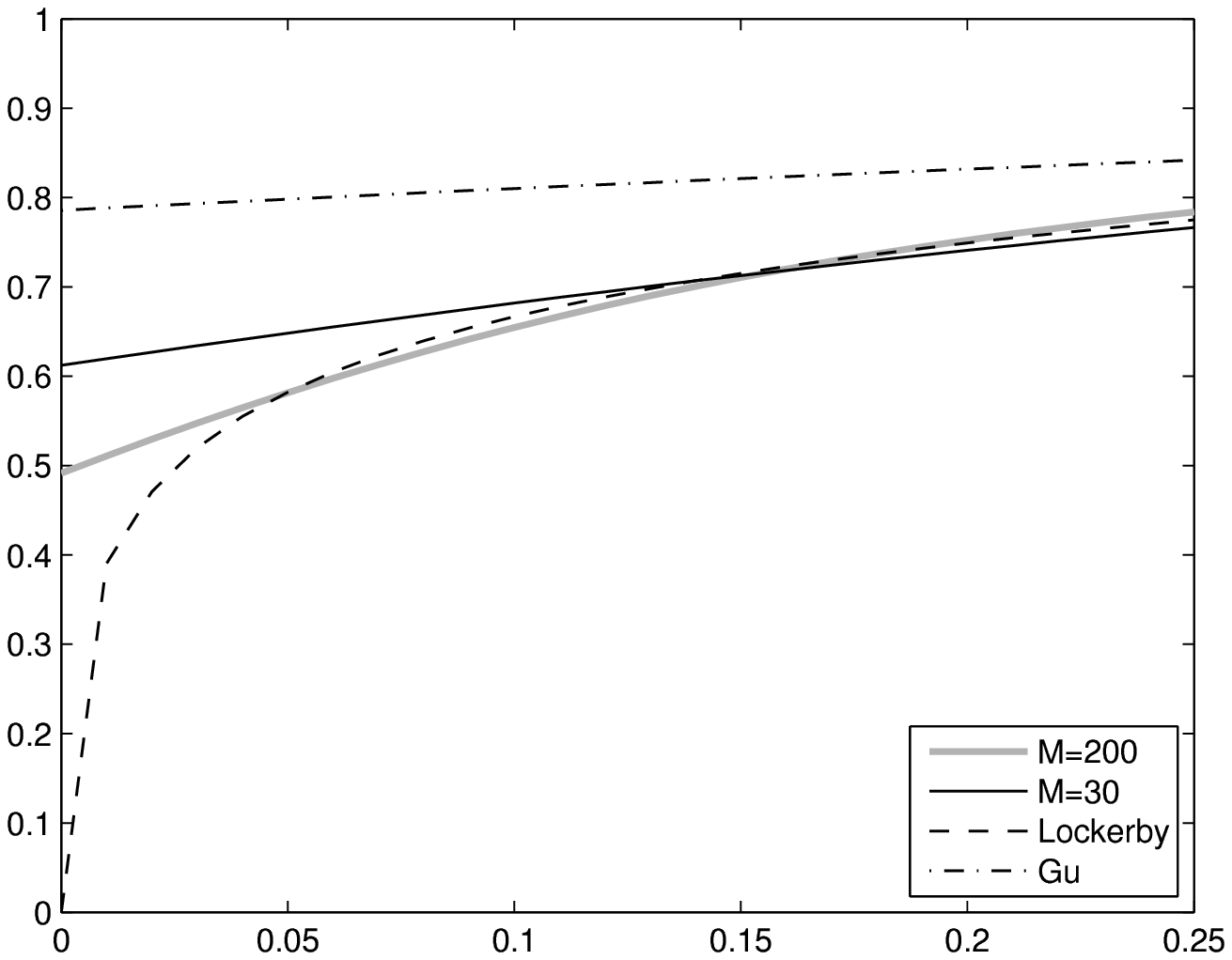}
            \put(-7,42){\small$\frac{\mu_{\mathrm{eff}}}{\mu}$}
            \put(60,-4){\small$\bar{y}$}
      \end{overpic}}
  \caption{Effective viscosity $\mu_{\mathrm{eff}}$ with different
    kinetic model.} 
  \label{fig:eff_vis}
\end{figure}

\begin{figure}[htb]
  \centering{
    \begin{overpic}[width=0.45\textwidth,clip]{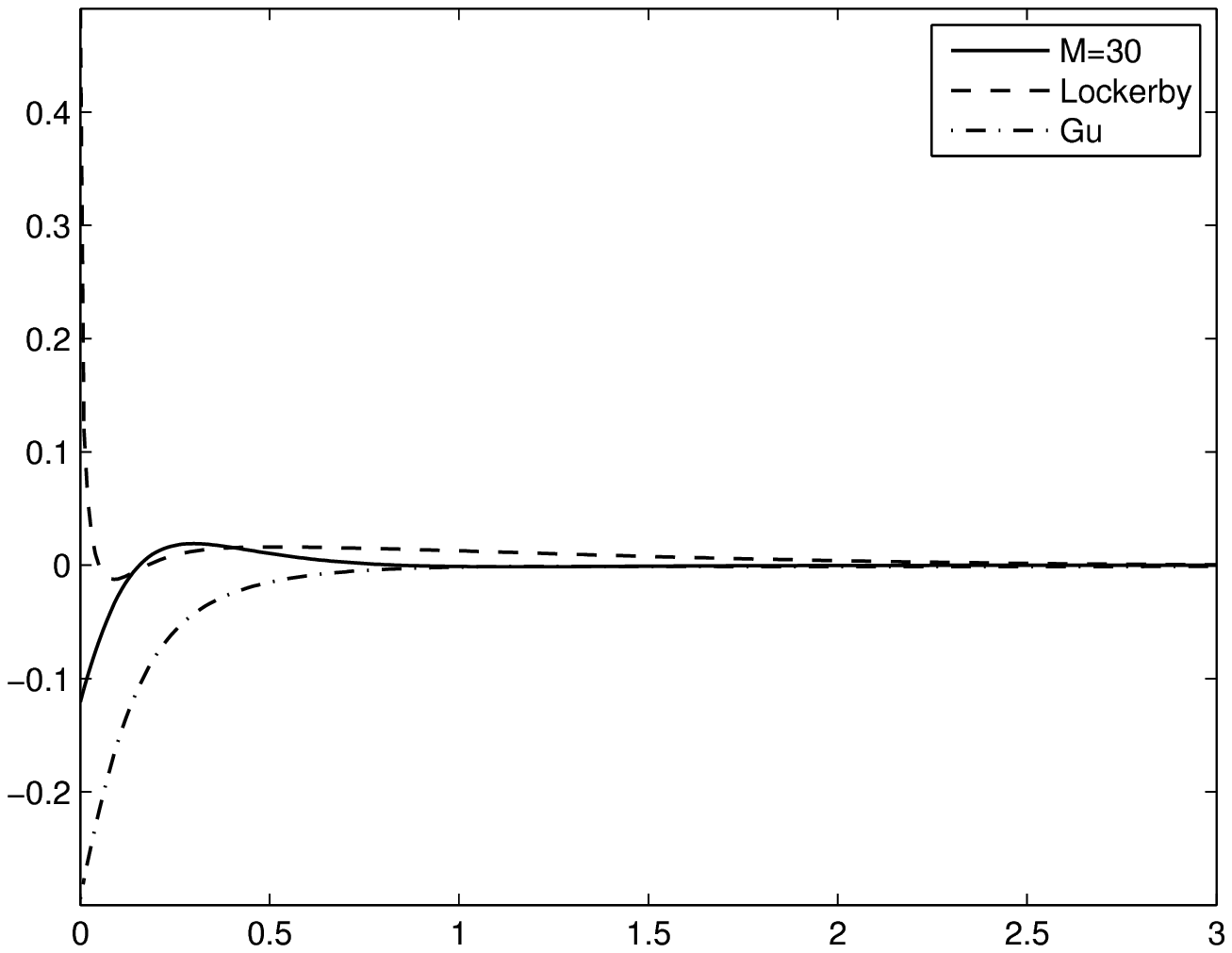}
        \put(-7,42){\small$\mathrm{err}$}
        \put(60,-4){\small$\bar{y}$}
      \end{overpic}}
  \caption{Comparison between effective viscosity $\mu_{\mathrm{eff}}$ with different
    kinetic model.} 
  \label{fig:eff_vis_err}
\end{figure}

In the past, the effective viscosity is well studied. For example, in
\cite{Gu_R26}, Gu investigated the R26 moment equations and predicted
the effective viscosity as
\begin{equation}\label{eq:eff_Gu}
    \mu_{\mathrm{eff}} = \left[1
        -\left(1.3042C_1 \exp\left(-\frac{1.265\bar{y}}{\Kn}\right) 
        + 1.6751C_2 \exp\left(-\frac{0.5102\bar{y}}{\Kn}\right) \right)
    \right]^{-1} \mu,
\end{equation}
where
\[
\begin{aligned}
C_1 = \frac{\chi -2}{\chi} \frac{0.81265 \times 10^{-1}\chi^2 +
  1.2824\chi}{0.48517 \times 10^{-2} \chi^2 + 0.64884 \chi +
  8.0995},\\
C_2 = \frac{\chi -2}{\chi} \frac{0.8565 \times 10^{-3}\chi^2 +
  0.362 \chi}{0.48517 \times 10^{-2} \chi^2 + 0.64884 \chi +
  8.0995}.
\end{aligned}
\]
This model is similar as the linearized HME. Actually, since R26
moment system can be derived from HME, Gu's result can be treated as a
special case of the linearized HME.
In \cite{Lockerby2008}, Lockerby et al. studied the effective
viscosity based on the two low-$\Kn$ BGK results, and proposed an
empirical expression as
\begin{equation}\label{eq:vis_Lockerby}
    \mu_{\mathrm{eff}} = \left(1 
    + 0.1859\bar{y}^{-0.464} \exp\left(-0.7902\bar{y}\right)
    \right)^{-1} \mu.
\end{equation}
For Lockerby's model, we have $\mu_{\mathrm{eff}}\to0$ as $y\to 0+$,
which indicates the velocity gradient to approach infinity at the
wall.

For convenience, here we let $\chi=1$ and $\Kn=1/\sqrt{2}$.  Fig.
\ref{fig:eff_vis} shows the profile of the effective viscosity of
these models. Due to the convergence of the linearized HME, we take
the solution the linearized HME with $M=200$ as the reference
solution. One can observe that Gu gives a relative larger effective
viscosity $\mu_{\mathrm{eff}}$, while Lockerby gives a relative smaller
one. If one want to obtain a good approximation of the effective
viscosity close to the wall, a lot of moments are needed.

We also take the solution of the linearized HME with $M=200$ as the
reference solution, and define the error as
\[
    \mathrm{err} = \mu_{\mathrm{eff}}^{\mathrm{reference}} - \mu_{\mathrm{eff}}^{\mathrm{model}}.
\]
Fig. \ref{fig:eff_vis_err} shows the error of Gu's and Lockerby's
models and the linearized HME with $M=30$.  Gu's model agrees with the
reference very well away from the Knudsen layer and gives too large
effective viscosity, while Lockerby's model gives too small effective
viscosity. For the linearized HME, by choosing a proper $M$, the
effective viscosity can be well captured.

\section{Conclusion}
In this paper, the globally hyperbolic moment equations (HME) is
employed to study Kramers'
problem. Firstly, the set of linearized globally hyperbolic moment
equations and their boundary conditions are built. The analytical
solutions for the defect velocity and slip coefficient have been
obtained for arbitrary order moment equations. In comparison with
data from kinetic theory, it has been shown that they can accurately 
capture the Knudsen layer velocity profile over a wide range of
accommodation coefficients, especially for the small accommodation
coefficients case. The results indicate that the physics of
non-equilibrium gas flow can be captured by high-order HME system.

\section{Acknowledgement}
The research of J. Li is partially supported by the Hong Kong Research
Council ECS grant No. $509213$ during her visit periods at the Hong
Kong Polytechnic University. The research of Y.-W. Fan, J. Li and R.
Li are supported by the National Natural Science Foundation of China
($11325102$, $11421110001$, $91630310$). The research of Z.-H. Qiao is
partially supported by the Hong Kong Research Council ECS grant No.
$509213$ and the Hong Kong Polytechnic University research fund
G-YBKP.



\bibliographystyle{plain}
\bibliography{article}
\end{document}